\documentclass{amsart}
\usepackage{amsmath,amsfonts,amssymb}
\usepackage{amsthm}
\usepackage{ifthen}
\usepackage{longtable}
\usepackage{array}
\usepackage{url}
\usepackage[utf8]{inputenc}
\usepackage{multirow}

\usepackage{color}

\newcommand{\Z}{\mathbb{Z}}
\newcommand{\Q}{\mathbb{Q}}

\newcommand{\ord}{\mathrm{ord}}

\newtheorem{definition}{Definition}
\newtheorem*{theorem*}{Theorem}
\newtheorem{theorem}{Theorem}
\newtheorem{lemma}{Lemma}
\newtheorem{corollary}{Corollary}
\newtheorem{proposition}{Proposition}
\newtheorem{conjecture}{Conjecture}
\newtheorem{problem}{Problem}

\newtheorem{claim}{Claim}
\newtheorem{algo}{Algorithm}

\theoremstyle{remark}
\newtheorem{remark}{Remark}

\begin{document}
\title[$S$-Diophantine quadruples]{On the existence of $S$-Diophantine quadruples}

\author[V. Ziegler]{Volker Ziegler}
\address{V. Ziegler,\\
Institute of Mathematics,\\
University of Salzburg,\\
Hellbrunnerstrasse 34/I,\\
A-5020 Salzburg, Austria}
\email{volker.ziegler\char'100sbg.ac.at}

\begin{abstract}
Let $S$ be a set of primes. We call an $m$-tuple $(a_1,\ldots,a_m)$ of distinct, positive integers $S$-Diophantine, if for all $i\neq j$ the integers $s_{i,j}:=a_ia_j+1$
have only prime divisors coming from the set $S$, i.e. if all $s_{i,j}$ are $S$-units. In this paper, we show that no $S$-Diophantine quadruple (i.e.~$m=4$) exists if
$S=\{3,q\}$. Furthermore we show that for all pairs of primes $(p,q)$ with $p<q$ and $p\equiv 3\mod 4$ no $\{p,q\}$-Diophantine quadruples exist, provided
that $(p,q)$ is not a Wieferich prime pair.
\end{abstract}


\maketitle

\section{Introduction}\label{Sec:Intro}

An $m$-tuple $(a_1,\ldots,a_m)$ of positive, distinct integers is called Diophantine, if 
\begin{equation}\label{Eq:Dioph}
a_ia_j+1=\square
\end{equation}
for $i\not=j$. Diophantine $m$-tuples have been studied since ancient times by
several authors. It was for a long time an open problem whether Diophantine quintuples exist and many mathematicians such as Fermat, Euler and in modern times
Baker and Davenport \cite{Baker:1969}, Peth\H{o} and Dujella \cite{Dujella:1998}
and Dujella \cite{Dujella:2004} investigated this problem. Only recently this problem was finally settled by He et.al.~\cite{He:2018} who showed that no Diophantine quintuples exist.

Also several variants of Diophantine tuples were considered by several authors. For instance, Bugeaud and Dujella \cite{Bugeaud:2003} considered the case where $\square$ is replaced by a $k$-th power in
\eqref{Eq:Dioph} and Luca and Szalay \cite{Luca:2008} considered the case where $\square$ is replaced by Fibonacci numbers. For an overview of all variants see \cite{Dujella:HP}.
In this paper we consider the following variant of Diophantine tuples. Let $S$ be a fixed (usually finite) set of primes. We call an $m$-tuple $(a_1,\ldots,a_m)$ of
distinct, positive integers $S$-Diophantine, if for all $i\neq j$ the integers $s_{i,j}:=a_ia_j+1$
have only prime divisors coming from the set $S$, i.e. if all $s_{i,j}$ are $S$-units. In view of classical Diophantine tuples the following question arises:

\begin{problem}
Let $S$ be a fixed finite set of primes. How large can a $S$-Diophantine tuple get? 
\end{problem}

This question has already been studied in a series of papers by Szalay and the author \cite{Szalay:2013a,Szalay:2013b,Szalay:2015}
and it is planed to continue this investigations in this paper.

In a slightly other context this problem was already studied by 
Gy\H{o}ry, S\'ark\"ozy and Stewart~\cite{Gyory:1996} who considered products of the form
\[\Pi(A,B)=\prod_{a\in A, b\in B}(ab+1),\]
where $A$ and $B$ are given sets of positive integers. They found lower bounds for the number of prime factors of $\Pi(A,B)$ in terms of $|A|$ and $|B|$. In particular, they showed that the number of
prime factors of $\Pi(A,A)$ exceeds $C \log |A|$, where $C$ is a positive, effectively computable constant, provided that $|A|\geq 3$. They also conjectured that the largest prime factor of
$(ab+1)(ac+1)(bc+1)$
tends to infinity as $\max\{a,b,c\}\rightarrow \infty$. A weaker form, namely that the largest prime factor of
$$(ab+1)(ac+1)(bd+1)(cd+1)$$
tends to infinity as $\max\{a,b,c,d\}\rightarrow \infty$, was proved by Stewart and Tijdeman \cite{Stewart:1997} and the full conjecture was proved independently 
by Corvaja and Zannier \cite{Corvaja:2003} and Hern{\'a}ndez and Luca \cite{Hernandez:2003}. In the context of $S$-Diophantine tuples the results of Corvaja, Zannier, Hern{\'a}ndez and Luca imply
that for a fixed, finite set of primes $S$ only finitely many $S$-Diophantine triples exist.

Of course for large, finite sets of primes $S$ also $S$-Diophantine $m$-tuples will exist for large $m$. Thus the following function introduced in \cite{Szalay:2013a} is of
special interest. Let $s(k)$ be the smallest integer $m$ such that for all sets of primes $S$ with $|S|=k$ no $S$-Diophantine $m$-tuple
exists. The results due to Gy\H{o}ry et.al. \cite{Gyory:1996} implies that such an $m$ exists for all $k$. In particular, their result \cite[Theorem 1 resp. Corollary 2]{Gyory:1996} yields the upper bound
$s(k)<\exp(Ck)$, where $C$ is an effectively computable absolute constant. 
On the other hand it is easy to show that $s(1)=3$ (see e.g. Lemma~\ref{lem:tripel} below). But, the exact values for $s(2)$ or any other $s(k)$ are yet unknown.
However we conjecture that $s(2)=4$. In other words we conjecture:

\begin{conjecture}\label{con:S-quadruple}
 Let $p<q$ be primes and $S=\{p,q\}$. Then no $S$-Diophantine quadruple exists.
\end{conjecture}

The author together with Szalay have solved several instances of this conjecture in a series of papers \cite{Szalay:2013a,Szalay:2013b,Szalay:2015}. In particular, they proved:

\begin{theorem*}[Szalay and Ziegler]
Let $S=\{p,q\}$ be a set of two primes. Then the following holds:
\begin{itemize}
 \item If $p^2\nmid q^{\ord_p(q)}-1$, $q^2\nmid p^{\ord_q(p)}-1$, and $q<p^\xi$ holds for
some $\xi>1$. Then there exists an effectively computable constant $C=C(\xi)$ such that for all such primes $p,q>C$
no $S$-Diophantine quadruple exists (see \cite{Szalay:2013a}).
 \item No $S$-Diophantine quadruple exists, if $p\equiv q\equiv 3 \mod 4$ (see \cite{Szalay:2013b}).
 \item No $S$-Diophantine quadruple exists, if $p=2$ and $q\equiv 3 \mod 4$ (see \cite{Szalay:2015}).
 \item No $S$-Diophantine quadruple exists, if $p=2$ and $q<10^9$  (see \cite{Szalay:2015}).
 \item No $S$-Diophantine quadruple exists, if $p,q<10^5$ (see \cite{Szalay:2015}).
\end{itemize}
\end{theorem*}

The next step toward proving Conjecture \ref{con:S-quadruple} is to prove the conjecture for small but fixed $p$. For instance, let $S=\{3,q\}$ and we may ask whether there exists
such a $S$-Diophantine quadruple. Indeed we can show the following.

\begin{theorem}\label{th:p23}
 Let $q>3$ be a prime and $S=\{3,q\}$ or $S=\{2,q\}$. Then no $S$-Diophantine quadruple exists.
\end{theorem}

Unfortunately the cases $p=2,3$ are somehow special and with our current method we cannot extend Theorem \ref{th:p23} to other fixed primes $p$. However as it was shown in \cite{Szalay:2013b}
the case that $p\equiv q \equiv 3 \mod 4$ is
rather easy. Therefore it is reasonable to investigate the case that either $p\equiv 3 \mod 4$ or $q\equiv 3 \mod 4$. Unfortunately we can exclude the existence of $S$-Diophantine quadruples in this case
only under the additional assumption that $p$ and $q$ form a Wieferich pair\footnote{That is $p^2|q^{p-1}-1$ and $q^2|p^{q-1}-1$. Note that some authors call this a double Wieferich prime pair.}.
Indeed even a weaker form is sufficient.

\begin{definition}
Let $p<q$ be primes. We call $(p,q)$ an extreme Wieferich pair if
\begin{equation}\label{eq:Wieferich-ext}
 v_q(p^{q-1}-1)\geq 2 \qquad \text{and} \qquad v_p(q^{p-1}-1)\geq \max\left\{2,\frac {\log q}{\log p}\right\},
\end{equation}
where $v_p(x)$ and $v_q(x)$ denote the $p$-adic and $q$-adic valuation of $x$, respectively.
\end{definition}

It is obvious that every extreme Wieferich pair is also an ordinary Wieferich pair, i.e. a pair of primes satisfying $v_q(p^{q-1}-1)\geq 2$ and
$v_p(q^{p-1}-1)\geq 2$. It is also obvious that in case that $p<q<p^2$ every ordinary Wieferich pair is also extreme. Up to now there are only seven known Wieferich pairs $(p,q)$,
but non of which is extreme and satisfies $q>p^2$.

However with this notation we can prove the following theorem: 

\begin{theorem}\label{th:gen}
 Let $p<q$ be primes and assume that $p\equiv 3 \mod 4$. Then a $\{p,q\}$-Diophantine quadruple exists only if $(p,q)$ is an extreme Wieferich pair, i.e. satisfies \eqref{eq:Wieferich-ext}.
\end{theorem}

 In other words if $p<q$, $p\equiv 3\mod 4$ and if $(p,q)$ is not an extreme Wieferich pair, then no $\{p,q\}$-Diophantine quadruple exists.

\section{Strategy of the paper}\label{Sec:strategy}

Before we explain the strategy of the paper we start with some notations. Therefore let $S=\{p,q\}$ be a set of two distinct primes with $p<q$ and let $(a,b,c,d)$ be a hypothetical $S$-Diophantine quadruple
satisfying $a<b<c<d$. These two assumptions will be kept throughout the paper and it will be stated explicitly if we do not assume them (this happens mainly in Section \ref{Sec:Prime} where we will drop the
assumption that $p<q$ holds). Then we write:
\begin{align*}
 ab+1=&p^{\alpha_1}q^{\beta_1}:=s_1, &  bc+1=&p^{\alpha_4}q^{\beta_4}:=s_4, \\
 ac+1=&p^{\alpha_2}q^{\beta_2}:=s_2, &  bd+1=&p^{\alpha_5}q^{\beta_5}:=s_5, \\
 ad+1=&p^{\alpha_3}q^{\beta_3}:=s_3, &  cd+1=&p^{\alpha_6}q^{\beta_6}:=s_6.
\end{align*}
Moreover, we let $A=\max_{i=1,\dots,6}\{\alpha_i\}$ and $B=\max_{i=1,\dots,6}\{\beta_i\}$. If we compute 
$$abcd=(s_1-1)(s_6-1)=(s_2-1)(s_5-1)=(s_3-1)(s_4-1)$$
in different ways we obtain the following three non-linear $S$-unit equations
\begin{equation}\label{eq:SUnit}
 \begin{split}
  s_1s_6-s_1-s_6=&s_2s_5-s_2-s_5,\\
  s_3s_4-s_3-s_4=&s_2s_5-s_2-s_5,\\
  s_1s_6-s_1-s_6=&s_3s_4-s_3-s_4.
 \end{split}
\end{equation}
A thorough study of this system of $S$-unit equations will yield Theorems \ref{th:p23} and \ref{th:gen}. Let us give a rough overview of the ideas that allow us to derive our main results from \eqref{eq:SUnit}.

In the next section we will gather all auxiliary results which are essential in proving our main Theorems \ref{th:p23} and \ref{th:gen}. In Section \ref{Sec:UB} we will prove upper bounds for the exponents
$\alpha_1,\dots,\alpha_6$ and $\beta_1,\dots,\beta_6$. In particular, the main result of Section \ref{Sec:UB} is
that 
$$\max\{A\log p ,B\log q\}< 52038 \log p\log q,$$
provided that $p\equiv 3\mod 4$ or $q\equiv 3\mod 4$. In Section \ref{Sec:Prime} we show that if $p\equiv 3 \mod 4$ or $q\equiv 3\mod 4$, then the exponents $\alpha_1,\dots,\alpha_6$ and $\beta_1,\dots,\beta_6$ 
have to fulfill rather restrictive relations (see Table \ref{tab:p=3mod4}). These restrictions allow us to show in Section \ref{Sec:Wief_padic} that 
under the assumption that $p\equiv 3 \mod 4$ and $q$ is large, i.e. $q>700393$, no $S$-Diophantine quadruple exists, if the $p$-adic Wieferich condition
 $$v_p(q^{p-1}-1)< \max\left\{2,\frac {\log q}{\log p}\right\}$$
 is fulfilled. An almost immediate consequence of this result is that no $\{3,q\}$-Diophantine quadruple exists
 and with a little bit more effort we can also show that no $\{2,q\}$-Diophantine quadruple exists. This is subject to Section \ref{Sec:p=3}. In Section \ref{Sec:Wief_qadic}
 we are interested in the case that the $q$-adic Wieferich condition $v_q(p^{q-1}-1)=1$ is fulfilled. In particular,
 we show that no $S$-Diophantine quadruple exists, if the $q$-adic Wieferich condition is fulfilled, $p\equiv 3 \mod 4$ and $q$ is large, i.e. $q>700393$. This proves Theorem \ref{th:gen}
 in the case that $q$ is large and we are left with the case that $p<q\leq 700393$. However these finitely many instances can be resolved by an algorithm due to Szalay and Ziegler~\cite{Szalay:2015} and
 we will discuss the implementation of the algorithm in Section \ref{Sec:Small}. In the final section we discuss further possible results and open problems.

\section{Auxiliary results}\label{Sec:aux}

We start with some lemmas that have been established already in \cite{Szalay:2013a,Szalay:2013b,Szalay:2015}. We start with the following simple divisibility condition which was proved in \cite[Lemma 2.1]{Szalay:2013a}:

\begin{lemma}\label{lem:tripel}
 Let $(a,b,c)$ be a $S$-Diophantine triple, with $a<b<c$, then $s\nmid t$ with $s=ac+1$ and $t=bc+1$.
\end{lemma}

Let us note that Lemma \ref{lem:tripel} implies that $\{p\}$-Diophantine triples do not exist, i.e. $s(1)=3$.
An immediate consequence of Lemma \ref{lem:tripel} is that we can exclude the following relations between exponents:
\begin{gather*}
 \alpha_2=\alpha_4, \quad \alpha_3=\alpha_5, \quad \alpha_3=\alpha_6, \quad \alpha_5=\alpha_6\\
 \beta_2=\beta_4, \quad \beta_3=\beta_5, \quad \beta_3=\beta_6, \quad \beta_5=\beta_6.
\end{gather*}
On the other hand we have the following lemma (cf. \cite[Proposition 1]{Szalay:2013a} or \cite[Lemma 2.1]{Szalay:2015}). This lemma is obtained by considering the equations
of system \eqref{eq:SUnit} and applying $p$-adic and $q$-adic valuations after we transformed them into a suitable form.

\begin{lemma}\label{lem:min_al}
  The smallest two exponents of the quadruples $(\alpha_2 , \alpha_3 , \alpha_4 , \alpha_5)$,
$(\alpha_1 , \alpha_2$, $\alpha_5 , \alpha_6)$ and $(\alpha_1 , \alpha_3, \alpha_4 , \alpha_6 )$ coincide, respectively. The same statement holds also 
with $\alpha$ replaced by $\beta$.
\end{lemma}

Also the following lemma proves to be useful and yields some elementary upper bounds for $a,b,c$ and $d$.

\begin{lemma}\label{lem:abcd-div}
We have
\begin{align*}
a&|\gcd\left(\frac{s_2-s_1}{\gcd(s_2,s_1)},\frac{s_3-s_1}{\gcd(s_3,s_1)},\frac{
s_3-s_2}{\gcd(s_3,s_2)}\right),\\
b&|\gcd\left(\frac{s_4-s_1}{\gcd(s_4,s_1)},\frac{s_5-s_1}{\gcd(s_5,s_1)},\frac{
s_5-s_4}{\gcd(s_5,s_4)}\right),\\
c&|\gcd\left(\frac{s_4-s_2}{\gcd(s_4,s_2)},\frac{s_6-s_2}{\gcd(s_6,s_2)},\frac{
s_6-s_4}{\gcd(s_6,s_4)}\right),\\
d&|\gcd\left(\frac{s_5-s_3}{\gcd(s_5,s_3)},\frac{s_6-s_3}{\gcd(s_6,s_3)},\frac{
s_6-s_5}{\gcd(s_6,s_5)}\right).
\end{align*}
\end{lemma}

\begin{proof}
 A proof can be found in \cite[Lemma 3]{Szalay:2013a}
\end{proof}

Another useful lemma is the following:

\begin{lemma}\label{lem:s4_low_bound}
 Let $(a, b, c, d) \in \Z^4$ be an $S$-Diophantine quadruple, such that $a<b<c<d$. Then
\begin{align*}
\gcd(s_4 , s_2 ) \gcd(s_4 , s_1 ) &< s_4,\\
\gcd(s_5 , s_3 ) \gcd(s_5 , s_1 ) &< s_5,\\
\gcd(s_6 , s_3 ) \gcd(s_6 , s_2 ) &< s_6,\\
\gcd(s_6 , s_5 ) \gcd(s_6 , s_4 ) &< s_6.
\end{align*}

\end{lemma}

\begin{proof}
 A proof can be found in \cite[Lemma 4]{Szalay:2013a} for the first inequality. By adjusting the proof it is easy to obtain the other inequalities.
\end{proof}

The next lemma can be seen as a summary of the results obtained in \cite[Sections 2 and 3]{Szalay:2013b} (for the general case) and \cite[Section 2]{Szalay:2015} (for the special case $p=2$):

\begin{lemma}\label{lem:quad_res}
 Let $p,q$ be primes (not necessarily $p<q$) and assume that $p\equiv 3 \mod 4$ (resp. that $p=2$). Then one of the following statements holds:
 \begin{itemize}
  \item $\alpha_1=\alpha_6=0$ (resp. $\alpha_1=\alpha_6= 1$),
  \item $\alpha_2=\alpha_5=0$ (resp. $\alpha_2=\alpha_5= 1$),
  \item $\alpha_3=\alpha_4=0$ (resp. $\alpha_3=\alpha_4= 1$).
 \end{itemize}
\end{lemma}

Unfortunately neither in \cite{Szalay:2013b} nor in \cite{Szalay:2015} the statement of Lemma \ref{lem:quad_res} is proved in this form. Therefore we give a proof for the sake of completeness.

\begin{proof}
 We start with the case that $p\equiv 3 \mod 4$. Let us assume for the moment that $\alpha_1,\alpha_2,\alpha_4>0$. Then we have
 $$(abc)^2=(p^{\alpha_1}q^{\beta_1}-1)(p^{\alpha_2}q^{\beta_2}-1)(p^{\alpha_4}q^{\beta_4}-1)\equiv -1 \mod p.$$
 Thus $-1$ is a quadratic residue modulo $p$ which is impossible since $p\equiv 3 \mod 4$. Thus at least one of
 $\alpha_1,\alpha_2$ or $\alpha_4$ is zero. Repeating the same argument with $(abc)^2$ replaced by $(abd)^2,(acd)^2$ and $(bcd)^2$ respectively we obtain the following statements:
 \begin{itemize}
  \item $0\in \{\alpha_1,\alpha_2,\alpha_4\}$,
  \item $0\in \{\alpha_1,\alpha_3,\alpha_5\}$,
  \item $0\in \{\alpha_2,\alpha_3,\alpha_6\}$,
  \item $0\in \{\alpha_4,\alpha_5,\alpha_6\}$.
 \end{itemize}
In view of the first statement we distinguish between the three cases $\alpha_1=0$ (Case A), $\alpha_2=0$ (Case B) and $\alpha_4=0$ (Case C). 

First, assume that Case A holds, i.e. $\alpha_1=0$.
Then the third statement implies that
either $\alpha_2=0$ or $\alpha_3=0$ or $\alpha_6=0$. Note that if $\alpha_6=0$ we are done. Thus we may assume that $\alpha_6\neq 0$. Hence, either $\alpha_2=0$ or $\alpha_3=0$.
Assume for the moment that $\alpha_2=0$ and consider the fourth statement.
Since $\alpha_2=\alpha_4=0$ is not possible due to Lemma \ref{lem:tripel} we arrive at $\alpha_2=\alpha_5=0$ and we are done, again. Now, assume that $\alpha_3=0$ and we consider again the fourth statement.
Since $\alpha_3=\alpha_5=0$ is not possible due to Lemma \ref{lem:tripel} we arrive at $\alpha_3=\alpha_4=0$. This shows that assuming Case A implies the statement of the lemma.

Now we consider Case B, i.e. $\alpha_2=0$, and due to the previous paragraph we may assume that $\alpha_1\neq 0$. Considering the second statement we have that either $\alpha_3=0$ or $\alpha_5=0$.
In the case that $\alpha_2=\alpha_5$ we are done and therefore we may assume that $\alpha_3=0$. Let us consider the fourth statement. But since we already assume that $\alpha_2=\alpha_3=0$ none of
$\alpha_4$, $\alpha_5$ or $\alpha_6$ can be zero due to Lemma~\ref{lem:tripel}.

Thus we finally may assume that $\alpha_4=0$ (Case C), but $\alpha_1,\alpha_2 \neq 0$. Thus due to the second statement we have that either $\alpha_4=\alpha_3=0$ or $\alpha_4=\alpha_5=0$. In case that $\alpha_3=0$ we are
done and therefore we may assume that $\alpha_4=\alpha_5=0$ and $\alpha_3\neq 0$. But then the third statement yields a contradiction. This proves the lemma in the case that $p\equiv 3 \mod 4$.

In the case that $p=2$ we know by \cite[Lemma 2.5]{Szalay:2013b} that up to permutations the remainders of $(a,b,c,d)$ modulo 4 are $(1, 1, 3, 3)$. Let us assume that $(a,b,c,d)\equiv (1,1,3,3) \mod 4$,
then $s_1=ab+1\equiv s_6=cd+1\equiv 2 \mod 4$ while $s_i\equiv 0 \mod 4$ for $i=2,3,4,5$. Thus we obtain $\alpha_1=\alpha_6=1$ and $\alpha_2,\alpha_3,\alpha_4,\alpha_5>1$.
The other $5$ permutations of possible values of $(a,b,c,d)$ modulo $4$ yield the other cases. We leave this easy verification to the reader.
\end{proof}

In their proof that no $\{p,q\}$-Diophantine quadruple exists, provided that $p\equiv q \equiv 3 \mod 4$, Szalay and Ziegler showed that the following system has no solution (see \cite[Section 4]{Szalay:2013b})
\begin{equation}\label{eq:alpha_beta_zero}
\begin{split}
ab+1=q^{\beta_1},\phantom{p^{\alpha_1}}& \qquad\qquad bc+1=p^{\alpha_4}q^{\beta_4}, \\ 
ac+1=p^{\alpha_2},\phantom{q^{\beta_2}}& \qquad\qquad bd+1=p^{\alpha_5}, \\
ad+1=p^{\alpha_3}q^{\beta_3},& \qquad\qquad cd+1=q^{\beta_6}. 
\end{split}
\end{equation}
This was proved without the assumptions that $p<q$ and $a<b<c<d$. That is they proved the following lemma:

\begin{lemma}\label{lem:alpha_beta_zero}
 Let $p<q$ be odd primes and assume that $(a,b,c,d)$ is a $\{p,q\}$-Diophantine quadruple, with $a<b<c<d$. Then the following two statements 
$$ \alpha_1=\alpha_6=0\quad \text{or}\quad\alpha_2=\alpha_5=0\quad \text{or}\quad\alpha_3=\alpha_4=0$$
and
$$ \beta_1=\beta_6=0\quad \text{or}\quad\beta_2=\beta_5=0\quad \text{or}\quad\beta_3=\beta_4=0$$
cannot hold simultaneously.
\end{lemma}

\begin{proof}
The proof of the lemma is more or less the content of \cite[Section 4]{Szalay:2013b}. Also note that the case that $\alpha_*=\beta_*=0$ 
would imply that $s_*=1$, with $*\in\{1,\dots,6\}$, which also yields a contradiction. 
\end{proof}

For the next lemma let us introduce the following notation for a fixed pair of primes $(p,q)$, with $p<q$. We write
$$ u_p=v_p(q^{p-1}-1) \quad \text{and} \quad u_q=v_q(p^{q-1}-1).$$

\begin{lemma}\label{lem:p-adic}
Let $p,q$ be odd primes (not necessarily $p<q$) and assume that $z=v_q(p^x-1)$. Then $z\leq u_q +\frac{x}{\log q}$. Moreover, if $2|\ord_q(p)$ and $z=v_q(p^x+1)$, then $z\leq u_q +\frac{x}{\log q}$.
If $2\nmid \ord_q(p)$, then $z=0$.
\end{lemma}

\begin{proof}
The lemma is elementary and some related versions can be found in \cite[Lemma 2.1.22]{Cohen:NTI}. In particular, it is proved that $v_q(s^n-1)=v_q(s-1)+v_q(n)$ if $s\equiv 1 \mod q$.
Putting $s=p^{q-1}$ we obtain the first statement of the Lemma by noting that $v_q(n)<\frac{\log n}{\log q}$. 

To prove the second statement we note that $\Z/q\Z$ is cyclic and therefore $p^x\equiv -1 \mod q$ holds if and only if $\ord_q(p)$ is even and $\frac{\ord_q(p)}2 |x$.
The second statement is now obtained by a slight modification of the proof given in \cite[Lemma 2.1.22]{Cohen:NTI}.
\end{proof}

One can see Lemma \ref{lem:p-adic} as a lower bound for a very special linear form of two $q$-adic logarithms (cf. Yamada's work \cite{Yamada:2010} on upper bounds for $v_p(x^{p-1}-1)$).
However, we will also use lower bounds for linear forms in complex logarithms.
In particular, we will apply the very sharp bounds due to Laurent \cite{Laurent:2008} for linear forms in two logarithms:

\begin{lemma}[Laurent 2008]\label{lem:twologs}
Assume that $\gamma_1,\gamma_2$ be two positive, real, multiplicatively
independent elements in a number field of degree $D$ over $\Q$. Moreover, assume that also $\log \gamma_1$ and $\log \gamma_2$ are
positive and real. For $i=1,2$, let $a_i>1$ be a real numbers satisfying
$$ \log a_i \geq \max \{h(\gamma_i),|\log\gamma_i|/D,1/D\}.$$
Further, let $b_1$ and $b_2$ be two positive integers.
Define
$$b'=\frac {b_1}{D\log a_2}+\frac {b_2}{D\log a_1} \quad \text{and} \quad \log b=\max\left\{\log b'+0.38,12/D,1\right\}.$$
Then
$$ |b_2 \log\gamma_2-b_1 \log\gamma_1|\geq \exp \left(-23.4 D^4(\log b)^2 \log a_1 \log a_2\right). $$
\end{lemma}

\begin{proof}
 Choose $m=12$ in \cite[Corollary 2]{Laurent:2008}. 
\end{proof}

The next lemma is part of the results derived in \cite{Szalay:2015}:

\begin{lemma}\label{lem:max_pq}
 If there exists a $\{p,q\}$-Diophantine quadruple, then $\max\{p,q\}>10^5$.
\end{lemma}

\begin{proof}
 This is part of Theorem 1.3 in \cite{Szalay:2015}.
\end{proof}

Finally we want to discuss the so-called $L$-notation (see also \cite[Section 3.1]{Heuberger:2004}) which is an exact form of the $O$-notation. 
Let $c$ be a real number, assume that $f(x),g(x)$ and $h(x)$ are real functions and $h(x)>0$ for $|x|<c$. We will write
$$f(x)=g(x)+L_c (h(x))$$
if
$$g(x)-h(x)\leq f(x) \leq g(x)+h(x)$$
holds for all $|x|<c$. The use of the $L$-notation is like the use of the $O$-Notation but with the advantage to have an explicit bound for the error term.
The following lemma is obvious from the definition of the $L$-notation and the geometric series expansion.

\begin{lemma}\label{lem:L-geometric}
 For some integer $n\geq 0$ and some real number $0<c<1$ we have that
 $$\frac{1}{1-x}=1+\dots+x^n+L_{c}\left(\frac{1}{1-c} x^{n+1}\right).$$
\end{lemma}

In all instances we will apply Lemma \ref{lem:L-geometric} only in the case that $x$ is of the form $x=\frac 1{p^\alpha q^\beta}$, with $\beta>0$. In view of Lemma \ref{lem:max_pq}
we have that $x<10^{-5}$ and by dropping the index of the $L$-notation we have 
$$\frac{1}{1-x}=1+x+L(1.001 x^2).$$ 

\section{An upper bound for the exponents}\label{Sec:UB}

The purpose of this section is to derive sharp upper bounds for $A$ and $B$ under the assumption that at least one of $p$ and $q$ is $\equiv 3 \mod 4$. However most
of the intermediate results of this section remain true for general prime pairs $(p,q)$. To be as general as possible we formulate and prove these intermediate results
without assuming that $p$ or $q$ is $\equiv 3 \mod 4$. However we start with the following lemma:

\begin{lemma}\label{lem:ub_beta12}
 Let $p<q$ and assume that a $\{p,q\}$-Diophantine quadruple exists. Moreover let
 $$A_1:=24.92 \left(\log \left(\frac{4.001A}{\log q}\right)\right)^2 \quad \text{and} \quad B_1:=24.92 \left(\log \left(\frac{4.001B}{\log p}\right)\right)^2.$$
 If we assume that $B>27826 \log p$, then we have 
 $$ \beta_1,\beta_2<B_1 \log p \quad \text{and} \quad \alpha_1,\alpha_2<B_1 \log q.$$
 If we assume that $A>27826 \log q$, then we have 
 $$ \beta_1,\beta_2<A_1 \log p \quad \text{and} \quad \alpha_1,\alpha_2<A_1 \log q.$$
\end{lemma}

\begin{proof}
 We will only prove the first statement, since the second statement is obtained by exchanging the roles of $p$ and $q$. 
 
 As already Stewart and Tijdeman \cite{Stewart:1997} observed, estimating the quantity $T=\frac{s_1s_6}{s_3s_4}$ proves to be useful. In particular, we obtain
 \begin{equation}\label{eq:Stewart}
 T=p^{A'}q^{B'}=\frac{(ab+1)(cd+1)}{(ad+1)(bc+1)}=1+\frac{(d-b)(c-a)}{abcd+ad+bc+1}<1+\frac{1}{ab}\leq 1+\frac 32 \cdot \frac{1}{ab+1}
 \end{equation}
 with $A'=\alpha_1+\alpha_6-\alpha_3-\alpha_4$ and $B'= \beta_1+\beta_6-\beta_3-\beta_4$. Moreover, we know that
 $ab\geq 2$ and therefore $1/2<p^{A'}q^{B'}<3/2$. This implies the following inequalities:
$$|A'|\leq \frac{2B \log q+\log(3/2)}{\log p}\quad \text{and} \quad |B'|\leq 2B.$$  
 We want to apply Lemma \ref{lem:twologs} with
 \begin{gather*}
  \gamma_1=p, \quad \gamma_2=q, \quad b_1=|A'|, \quad b_2=|B'|,\\
  D=1, \quad \log a_1=\log p, \quad \log a_2=\log q,
 \end{gather*}
 and therefore we estimate
 $$ b'=\frac {b_1}{D\log a_2}+\frac {b_2}{D\log a_1}\leq \frac{4B}{\log p}+ \frac{\log (3/2)}{\log p \log q}< \frac{4B}{\log p}+0.117.$$
 Now Lemma \ref{lem:twologs} yields
 \begin{equation}\label{eq:lin_form_ieq_I}
 \begin{split}
  \log |\log T|>&-23.4 \log p \log q \left( \log\left(\frac{4B}{\log p}+0.117\right)+0.38\right)^2\\
  >&- 24.91\log p \log q \left( \log \left(\frac{4.001 B}{\log p}\right)\right)^2
  \end{split}
  \end{equation}
 provided that $\log\left(\frac{4B}{\log p}+0.117\right)+0.38>12$, i.e. that $B>27826 \log p$.
 
 On the other hand we know that $|\log(1+x)|<2|x|$ provided that $|x|<1/2$ and therefore inequality \eqref{eq:Stewart} yields
 \begin{equation}\label{eq:lin_form_ieq_II}
 \log |\log T|< \log\left(\frac{3}{p^{\alpha_1}q^{\beta_1}}\right)<\log 3 - \beta_1 \log q. 
 \end{equation}
Comparing the lower bound \eqref{eq:lin_form_ieq_I} with the upper bound \eqref{eq:lin_form_ieq_II} we obtain
$$\beta_1<24.91 \log p \left( \log \left(\frac{4.001B}{\log p}\right)\right)^2+\frac{\log 3}{\log q}<24.92 \log p \left( \log \left(\frac{4.001B}{\log p}\right)\right)^2.$$

We obtain the inequality for $\beta_2$ by considering instead of $T$ the quantity
$$
 T'= p^{A''}q^{B''}=\frac{(ac+1)(bd+1)}{(ad+1)(bc+1)}=1+\frac{(d-c)(b-a)}{abcd+ad+bc+1}<1+\frac{1}{ac}\leq 1+\frac 32 \cdot\frac{1}{ac+1}
$$
with $A''=\alpha_2+\alpha_5-\alpha_3-\alpha_4$ and $B''=\beta_1+\beta_6-\beta_3-\beta_4$. Similar computations as above yield the same upper bound for $\beta_2$.

The upper bound for $\alpha_1$ is obtained by noting that instead of \eqref{eq:lin_form_ieq_II} one can use the upper bound
$$ \log |\log T|<\log\left(\frac{3}{p^{\alpha_1}q^{\beta_1}}\right)<\log 3 - \alpha_1 \log p. $$
 A slight modification finally yields an upper bound for $\alpha_2$. 
\end{proof}

Our next step is to show that with $ab+1$ and $ac+1$ also $bc+1$ stays small. To be more precise we prove

\begin{lemma}\label{lem:ub_beta4}
 Let $p<q$ and assume that a $\{p,q\}$-Diophantine quadruple exists. 
 If we assume that $B>27826 \log p$, then we have
 $$\beta_4<4B_1 \log p \quad \text{and} \quad \alpha_4<4B_1 \log q.$$
 If we assume that $A>27826 \log q$, then we have 
 $$\beta_4<4A_1 \log p \quad \text{and} \quad \alpha_4<4A_1 \log q.$$
\end{lemma}

\begin{proof}
 We only prove the upper bound for $\beta_4$ in the case that $B>27826 \log p$. All other instances are obtained by slight modifications of the argument. However, the upper bound for $\beta_4$ is obtained by the
 following inequality
 $$  \exp(\beta_4 \log q) \leq  bc+1 <(ab+1)(bc+1) =p^{\alpha_1+\alpha_2}q^{\beta_1+\beta_2} <\exp(4 B_1 \log p \log q).$$
\end{proof}

The Lemmas \ref{lem:ub_beta12} and \ref{lem:ub_beta4} show that the exponents $\beta_1,\beta_2,\beta_4$ and $\alpha_1,\alpha_2,\alpha_4$ are rather small and therefore also the $S$-units $s_1,s_2,s_4$ stay small.
The next lemma shows that if one further $S$-unit is small, then all exponents are small.

\begin{lemma}\label{lem:Alt1_bound}
 Let $*\in\{3,5,6\}$ and assume that the following two inequalities hold
 \begin{itemize}
 \item $\beta_*\leq \max\{\beta_1,\beta_2,\beta_4\}$ and 
 \item $\alpha_*\leq \max\{\alpha_1,\alpha_2,\alpha_4\}$.
 \end{itemize}
 Then $B<34990 \log p$ and $A<34990 \log q$.
\end{lemma}

\begin{proof}
 We start with proving the inequality for $B$. Therefore we may assume that $B>27826 \log p$ and we obtain
 the inequality
 $$d<s_*=p^{\alpha_*}q^{\beta_*}<\exp(8B_1 \log p \log q).$$
 This implies
 $$q^B\leq cd+1<d(ac+1)<\exp(10B_1 \log p \log q)$$
 and we obtain the inequality
 $$B<249.2 \log p \left(\log \left(\frac{4.001B}{\log p}\right)\right)^2.$$
 Let us write $x=\frac{4.001B}{\log p}$, then we obtain the inequality $x<249.2\cdot 4.001 (\log x)^2$
 which yields $x<139993$. Thus we obtain $B<34990 \log p$.
 
 The inequality for $A$ is obtained by exchanging the roles of $p$ and $q$.
\end{proof}

The next lemma shows that there cannot be a single large exponent out of $\beta_1,\dots,\beta_6$ or $\alpha_1,\dots,\alpha_6$ respectively.

\begin{lemma}\label{lem:Alt2_bound}
 Let $*<\dagger \in \{3,5,6\}$ and assume that at least one out of the two following inequalities holds
 \begin{itemize}
  \item $\beta_*,\beta_\dagger\leq \max\{\beta_1,\beta_2,\beta_4\}$ or
  \item $\alpha_*,\alpha_\dagger\leq \max\{\alpha_1,\alpha_2,\alpha_4\}$.
 \end{itemize}
 Then $B< 52038\log p$ and $A<52038 \log q$.
\end{lemma}

\begin{proof}
 We only prove the inequality for $B$ since the inequality for $A$ can be shown by the same way of reasoning. 
 In view of the content of the lemma we may assume that $B\geq 27826 \log p$ and due to Lemmas \ref{lem:ub_beta12} and \ref{lem:ub_beta4} we have that
 \begin{align*}
  \alpha_1,\alpha_2&<B_1 \log q,& \alpha_4&<4B_1 \log q ,\\
  \beta_1,\beta_2&<B_1 \log p,& \beta_4&<4B_1 \log p .
 \end{align*}
 However, let us start with the following claim.
 
 \begin{claim}\label{claim:ieq}
  If $\beta_*,\beta_\dagger\leq \max\{\beta_1,\beta_2,\beta_4\}$, then 
  $$|\alpha_*-\alpha_\dagger|\leq 12 B_1 \log q +\frac{\log (3/2)}{\log p}.$$
  If $\alpha_*,\alpha_\dagger\leq \max\{\alpha_1,\alpha_2,\alpha_4\}$, then
  $$|\beta_*-\beta_\dagger|\leq 12 B_1 \log q +\frac{\log (3/2)}{\log q}.$$
  \end{claim}

\begin{proof}[Proof of Claim \ref{claim:ieq}]
We consider the quantity $T_*=\frac{s_*s_{7-*}}{s_\dagger s_{7-\dagger}}$ and obtain an estimate similar to inequality~\eqref{eq:Stewart}
$$T_*=p^{\alpha_*-\alpha_\dagger+\alpha_{7-*}-\alpha_{7-\dagger}}q^{\beta_*-\beta_\dagger+\beta_{7-*}-\beta_{7-\dagger}}<\frac 32 $$
which yields
$$p^{|\alpha_*-\alpha_\dagger|}<\frac 32 q^{\beta*+\beta_{7-*}} p^{|\alpha_{7-*}-\alpha_{7-\dagger}|}<\frac 32 \exp(12 B_1 \log p \log q )$$
and therefore we obtain the first claim. The second claim is obtained by a similar argument.
\end{proof}

We continue the proof of Lemma \ref{lem:Alt2_bound}. Claim \ref{claim:ieq} shows that in any case we have the inequality
$$\max\{|\alpha_*-\alpha_\dagger|\log q, |\beta_*-\beta_\dagger|\log p\}<12 B_1 \log p \log q+\log(3/2).$$

Next we observe that due to Lemma \ref{lem:tripel} we have $s_*\nmid s_\dagger$, that is either 
\begin{itemize}
 \item $\alpha_*<\alpha_\dagger$ and $\beta_\dagger<\beta_*$, or
 \item $\beta_*<\beta_\dagger$ and $\alpha_\dagger<\alpha_*$.
\end{itemize}
Now we apply Lemma \ref{lem:abcd-div} and obtain that $d|\frac{s_\dagger-s_*}{\gcd (s_\dagger,s_*)}$. Hence, we have
$$d\leq \frac{s_\dagger}{\gcd (s_\dagger,s_*)} =\max\left\{p^{\alpha_\dagger-\alpha_*},q^{\beta_\dagger-\beta_*}\right\}\leq \frac 32 \exp(12 B_1 \log p \log q).$$
But, we also have
$$q^B\leq cd+1<d(ac+1)< \frac 32 \exp(14 B_1 \log p \log q)$$
which yields the inequality
$$B< 348.89 \log p \left(\log \frac{4.001 B}{\log p}\right)^2.$$
If we substitute $x=\frac{4.001 B}{\log p}$, we obtain the inequality $x<4.001 \cdot 348.89 (\log x)^2$ which implies that $x< 209283$, hence $B<52038 \log p$.
\end{proof}

Lemmas \ref{lem:Alt1_bound} and \ref{lem:Alt2_bound} result in many restrictions on the exponents $\alpha_1,\dots,\alpha_6$ and $\beta_1,\dots, \beta_6$ and in combination with Lemmas \ref{lem:tripel} and \ref{lem:min_al}
it is possible to prove the main result of this section:

\begin{proposition}\label{prop:bound_AB}
 Let $p<q$ be primes and assume that there exists a $\{p,q\}$-Dioph\-antine quadruple. Then either
 $$\max\{A\log p,B\log q\} < 52038 \log p \log q$$
 or one of the following two cases holds:
 \begin{description}
  \item[Case I\phantom{I}] $\alpha_1=\alpha_2=\alpha_3<\alpha_4<\alpha_6<\alpha_5$ and $\beta_1=\beta_4=\beta_5<\beta_2<\beta_6<\beta_3$, or
  \item[Case II] $\alpha_1=\alpha_4=\alpha_5<\alpha_2<\alpha_6<\alpha_3$ and $\beta_1=\beta_2=\beta_3<\beta_4<\beta_6<\beta_5$.
 \end{description}
\end{proposition}

\begin{proof}
In view of the content of the proposition, we assume that 
$$\max\{A\log p,B\log q\} \geq 52038 \log p \log q$$
and we will show that either Case I or II holds.

We start by applying Lemma \ref{lem:min_al} and deduce that the two smallest exponents of the quadruple $(\beta_2,\beta_3,\beta_4,\beta_5)$ coincide. Since by Lemma \ref{lem:tripel} we may exclude
the case that $\beta_2=\beta_4$ and $\beta_3=\beta_5$ and in combination with Lemma \ref{lem:Alt2_bound} and the fact that $A$ or $B$ is large we are left
with the following four cases:
\begin{enumerate}
 \item $\beta_2=\beta_3<\beta_4<\beta_5$,
 \item $\beta_2=\beta_5<\beta_4<\beta_3$,
 \item $\beta_3=\beta_4<\beta_2<\beta_5$,
 \item $\beta_4=\beta_5<\beta_2<\beta_3$.
\end{enumerate}
Considering $\alpha$ instead of $\beta$ we obtain the same list of cases with $\beta$ replaced by $\alpha$. Let us have a closer look on each individual case.

We start with Case (1). Since Lemma \ref{lem:Alt1_bound} we deduce that $\alpha_3$ is not minimal. Moreover $\alpha_2<\alpha_4$ is not possible, since otherwise $s_2|s_4$ which
contradicts Lemma \ref{lem:tripel}. Therefore the only possibility left for the $\alpha$-exponents is $\alpha_4=\alpha_5<\alpha_2<\alpha_3$.

In Case (2) we conclude by Lemma \ref{lem:Alt1_bound} that $\alpha_5$ cannot be minimal and by Lemma \ref{lem:tripel} we have $\alpha_4<\alpha_2$. Thus $\alpha_4=\alpha_3<\alpha_2<\alpha_5$. Note that 
$\alpha_5\leq \alpha_2$ can be excluded due to Lemma \ref{lem:Alt2_bound}.

In Case (3) we have that $\alpha_3$ is not minimal (Lemma \ref{lem:Alt1_bound}), $\alpha_2<\alpha_4$ (Lemma \ref{lem:tripel}) and $\alpha_4<\alpha_3$ (Lemma \ref{lem:Alt2_bound}). Thus we obtain
$\alpha_2=\alpha_5<\alpha_4<\alpha_3$.

Finally in Case (4) we have that $\alpha_5$ is not minimal (Lemma \ref{lem:Alt1_bound}), $\alpha_2<\alpha_4$ (Lemma \ref{lem:tripel}) and $\alpha_4<\alpha_5$ (Lemma \ref{lem:Alt2_bound}). Thus we obtain
$\alpha_2=\alpha_3<\alpha_4<\alpha_5$. 

Therefore we have to distinguish between the following four cases:
\begin{enumerate}
 \item $\beta_2=\beta_3<\beta_4<\beta_5$ and $\alpha_4=\alpha_5<\alpha_2<\alpha_3$,
 \item $\beta_2=\beta_5<\beta_4<\beta_3$ and $\alpha_4=\alpha_3<\alpha_2<\alpha_5$,
 \item $\beta_3=\beta_4<\beta_2<\beta_5$ and $\alpha_2=\alpha_5<\alpha_4<\alpha_3$,
 \item $\beta_4=\beta_5<\beta_2<\beta_3$ and $\alpha_2=\alpha_3<\alpha_4<\alpha_5$. 
\end{enumerate}

Next, we apply again Lemma \ref{lem:min_al} and deduce that the two smallest exponents of the quadruple $(\beta_1,\beta_2,\beta_5,\beta_6)$ coincide. Note
that $\beta_6$ cannot be minimal because of Lemma \ref{lem:Alt2_bound}.
Otherwise two out of $\beta_3,\beta_5,\beta_6$ would be small in any of the above discussed cases. Therefore either $\beta_2=\beta_5$ or $\beta_1=\beta_2$ or $\beta_1=\beta_5$ is minimal.

Let us start with discussing the case that $\beta_2=\beta_5$ is minimal. The only case that admits $\beta_2=\beta_5$ is Case~(2). We want to apply Lemma \ref{lem:min_al} to the quadruple
$(\alpha_1,\alpha_2,\alpha_5,\alpha_6)$.
Since already $\alpha_3$ is small neither $\alpha_5$ nor $\alpha_6$ can be minimal, because otherwise
this would contradict Lemma \ref{lem:Alt2_bound}. Therefore we conclude that $\alpha_1=\alpha_2<\alpha_5,\alpha_6$ and in combination with Case (2) we obtain
$$ \beta_2=\beta_5<\beta_4<\beta_3 \quad \text{and} \quad \alpha_4=\alpha_3<\alpha_1=\alpha_2<\alpha_5,\alpha_6.$$
Once again we apply Lemma \ref{lem:min_al} and conclude that the two smallest exponents of the quadruple $(\beta_1,\beta_3,\beta_4,\beta_6)$ must coincide.
Since the minimality of $\beta_3$ or $\beta_6$ would yield that two exponents out
of $\beta_3,\beta_5,\beta_6$ would be small, we conclude that $\beta_1=\beta_4$. On the other hand we have that $\alpha_4<\alpha_1$ and therefore we get $s_4<s_1$ an obvious contradiction and the case that
$\beta_2=\beta_5$ is minimal cannot occur.

Next we consider the case that $\beta_1=\beta_2$ is minimal. Since the minimality of $\beta_2=\beta_5$ has been excluded in the previous paragraph, 
we deduce that $\beta_2<\beta_5$ and that either Case (1) or Case (3) holds.
Let us assume for the moment that Case (1) holds. Since $\beta_3$ is small we deduce that $\beta_6$ cannot be small because of Lemma \ref{lem:Alt2_bound} and we deduce that
$$ \beta_1=\beta_2=\beta_3<\beta_4<\beta_5,\beta_6 \quad \text{and} \quad \alpha_4=\alpha_5<\alpha_2<\alpha_3.$$
However Lemma \ref{lem:min_al} tells us that the two smallest exponents out of the quadruple $(\alpha_1,\alpha_2,\alpha_5,\alpha_6)$ must coincide.
Again Lemma \ref{lem:Alt2_bound} shows that $\alpha_6$ cannot be minimal. Therefore
we have either $\alpha_1=\alpha_2$ or $\alpha_1=\alpha_5$ or $\alpha_2=\alpha_5$ is minimal. But $\alpha_1=\alpha_2$ is impossible, since otherwise $s_1=s_2$, an obvious contradiction. Also $\alpha_2=\alpha_5$ is impossible, since
this would imply $\alpha_2=\alpha_4=\alpha_5$ and $s_2|s_4$ which contradicts Lemma~\ref{lem:tripel}. Therefore we have $\alpha_1=\alpha_5$ and we conclude that
$$ \beta_1=\beta_2=\beta_3<\beta_4<\beta_5,\beta_6 \quad \text{and} \quad \alpha_1=\alpha_4=\alpha_5<\alpha_2<\alpha_3,\alpha_6.$$
Since $s_3\nmid s_6$ and $s_5\nmid s_6$ we finally arrive at Case II of the proposition.

Now let us assume that $\beta_1=\beta_2$ is minimal and that Case (3) holds. Similar as in the paragraph above we deduce that
$$ \beta_1=\beta_2<\beta_3=\beta_4<\beta_5,\beta_6 \quad \text{and} \quad \alpha_2=\alpha_5<\alpha_4<\alpha_3.$$
But this immediately implies $s_2|s_4$, which contradicts Lemma \ref{lem:tripel}.
 
Therefore we are left with the case that $\beta_1=\beta_5$ is minimal. Since due to the previous cases we may exclude that $\beta_2=\beta_5$. Therefore only Case (4) may hold and we obtain
$$ \beta_1=\beta_4=\beta_5<\beta_2<\beta_3,\beta_6 \quad \text{and} \quad \alpha_2=\alpha_3<\alpha_4<\alpha_5.$$
Once again we utilize Lemma \ref{lem:min_al} and use the fact that the two smallest exponents of the quadruple $(\alpha_1,\alpha_2,\alpha_5,\alpha_6)$ must coincide. Since already $\alpha_3$ is small
neither $\alpha_5$ nor $\alpha_6$ can be minimal and we obtain that $\alpha_1=\alpha_2$ is minimal. Moreover, $\alpha_6<\alpha_4$ would contradict Lemma \ref{lem:Alt2_bound}. Putting everything together we obtain 
$$ \beta_1=\beta_4=\beta_5<\beta_2<\beta_3,\beta_6 \quad \text{and} \quad \alpha_1=\alpha_2=\alpha_3<\alpha_4<\alpha_5,\alpha_6.$$
By noting that $s_5\nmid s_6$ and $s_3\nmid s_6$ we obtain Case I of the proposition.
\end{proof}

In what follows the following consequence of Proposition \ref{prop:bound_AB} will be useful:

\begin{corollary}\label{cor:bound_AB}
 Let $p<q$ be primes not both $\equiv 1 \mod 4$ and assume that there exists a $\{p,q\}$-Diophantine quadruple. Then we have 
 $$\max\{A\log p,B\log q\} \leq 52038 \log p \log q.$$
\end{corollary}

\begin{proof}
 Let us assume that $p\not \equiv 1 \mod 4$. Then due to Lemma \ref{lem:quad_res} we have that either $\alpha_1=\alpha_6$ or $\alpha_2=\alpha_5$ or $\alpha_3=\alpha_4$. But, if also
 Case I or II of Proposition \ref{prop:bound_AB} holds, then this yields a contradiction to the content of Lemma \ref{lem:tripel}. Thus we obtain the claimed upper bound.
 
 In the case that $q\not \equiv 1 \mod 4$ a similar argument can be applied to obtain the corollary.
\end{proof}

In the case that $p=2$ or $p=3$ we immediately obtain the following absolute upper bounds for~$B$:

\begin{corollary}\label{cor:bound_23}
 Let $q\neq 2,3$ be a prime and assume that there exists a $\{2,q\}$-Diophantine quadruple or a $\{3,q\}$-Diophantine quadruple. 
 Then $B \leq 36070$ and $B\leq 57170$ respectively.
\end{corollary}

\begin{remark}
 We want to stress out that the key to obtain rather small upper bounds for $A$ and $B$ is that we can show that Cases I and II cannot hold unless $p\equiv q \equiv 1 \mod 4$.
 The author does not see how to obtain such small bounds in the case that $p\equiv q \equiv 1 \mod 4$. However, let us mention that due to the method of Stewart and Tijdeman \cite{Stewart:1997}
 one can show that 
  $$\max\{A\log p,B\log q\} \ll (\log p \log q)^3 \left(\log\log p+\log\log q\right)^4.$$
  This upper bound can be obtained by applying \cite[Lemma 7]{Szalay:2013a} (see also \cite[Section 4]{Szalay:2015}) together with a result due to Peth\H{o} and de Weger \cite{Pethoe:1986} on the upper bound for solutions to
 $x=u+v(\log x)^h$. 
\end{remark}

\section{The presence of a prime $\not\equiv 1 \mod 4$}\label{Sec:Prime}

In this section we want to derive some consequences from the fact that $p\equiv 3 \mod 4$. However to keep our results as general as possible we drop the assumption that
$p<q$ in this section, but keep the assumption that $a<b<c<d$. However, we prove the following proposition, which is also the main result of this section:

\begin{proposition}\label{prop:p=3mod4}
 Let $p,q$ be odd, distinct primes (not necessarily $p<q$), with $p\equiv 3 \mod 4$. If there exists a $\{p,q\}$-Diophantine quadruple, then one of the four cases in Table \ref{tab:p=3mod4} holds.
\end{proposition}

\begin{table}[ht]
\caption{Restrictions to the exponents}\label{tab:p=3mod4}
 \begin{tabular}{|c|l|l|}
 \hline Case & The $\alpha$ exponents & The $\beta$ exponents \\ \hline\hline
 I & $0=\alpha_1=\alpha_6<\alpha_4=\alpha_5<\alpha_2<\alpha_3$ & $\beta_1=\beta_2=\beta_3<\beta_4<\beta_5<\beta_6$ \\\hline
 II & $0=\alpha_1=\alpha_6<\alpha_2=\alpha_5<\alpha_3,\alpha_4$ & $\beta_3=\beta_4<\beta_1=\beta_2< \beta_5<\beta_6$ \\\hline
 III & $0=\alpha_2=\alpha_5<\alpha_1=\alpha_3\leq \alpha_4<\alpha_6$ & $\beta_1=\beta_6<\beta_3=\beta_4<\beta_2<\beta_5$ \\\hline
 IV & $0=\alpha_3=\alpha_4<\alpha_1=\alpha_2<\alpha_5<\alpha_6$ & $\beta_1=\beta_6<\beta_2=\beta_5< \beta_3,\beta_4$\\ \hline
 \end{tabular}
\end{table}

Since $p\equiv 3 \mod 4$ we may apply Lemma \ref{lem:quad_res} and we have to distinguish between the three cases $\alpha_1=\alpha_6=0$, $\alpha_2=\alpha_5=0$ and $\alpha_3=\alpha_4=0$.
We prove Proposition \ref{prop:p=3mod4} in each of these three individual cases.

\subsection{The case that $\alpha_1=\alpha_6=0$}

If $\alpha_1=\alpha_6=0$, then we deduce that $\beta_1<\beta_6$ and in view of Lemma~\ref{lem:min_al} applied to the quadruple $(\beta_1,\beta_2,\beta_5,\beta_6)$ we 
have to distinguish between the following three subcases:
\begin{description}
 \item[Case A] $\beta_1=\beta_2\leq \beta_5<\beta_6$;
 \item[Case B] $\beta_1=\beta_5\leq \beta_2<\beta_6$;
 \item[Case C] $\beta_2=\beta_5\leq \beta_1<\beta_6$.
\end{description}
Note that since $\alpha_6=0$ the exponent $\beta_6$ is the largest exponent among $\beta_1,\dots,\beta_6$.

\subsubsection{Case A} 
In this case we apply Lemma \ref{lem:min_al} to the quadruple $(\beta_1,\beta_3,\beta_4,\beta_6)$ and since $\beta_1<\beta_6$ we have that one of the following three options holds:
\begin{itemize}
 \item $\beta_1=\beta_3\leq \beta_4<\beta_6$;
 \item $\beta_1=\beta_4\leq \beta_3<\beta_6$;
 \item $\beta_3=\beta_4\leq \beta_1<\beta_6$.
\end{itemize}
The second option can be dismissed since otherwise we would have $\beta_1=\beta_2=\beta_4$ which implies $s_2|s_4$, a contradiction to Lemma \ref{lem:tripel}. Moreover,
we may assume for the other options that $\beta_1<\beta_4$.

Let us consider the first option. This yields
$$\beta_1=\beta_2=\beta_3<\beta_4,\beta_5<\beta_6 \quad \text{and}\quad 0=\alpha_1=\alpha_6.$$
Since Lemma \ref{lem:tripel} we cannot have $s_3|s_5$ and $s_2|s_4$, i.e. we have $\alpha_4<\alpha_2$ and $\alpha_5<\alpha_3$. Therefore Lemma \ref{lem:min_al} applied
to the quadruple $(\alpha_2,\alpha_3,\alpha_4,\alpha_5)$ we have $\alpha_4=\alpha_5<\alpha_2,\alpha_3$. Let us note that $\beta_4<\beta_5$, since $\alpha_4=\alpha_5$ and $s_4<s_5$ holds,
and let us note that $\alpha_2<\alpha_3$ since $\beta_2=\beta_3$ and $s_2<s_3$ holds. Putting everything together, we obtain 
$$0=\alpha_1=\alpha_6<\alpha_4=\alpha_5<\alpha_2<\alpha_3\quad \text{and} \quad\beta_1=\beta_2=\beta_3<\beta_4<\beta_5<\beta_6$$
and get Case I in Table \ref{tab:p=3mod4}.

The argument for the third case, i.e. for the case that
$$\beta_3=\beta_4<\beta_1=\beta_2\leq \beta_5<\beta_6 \quad \text{and}\quad 0=\alpha_1=\alpha_6 $$
is similar. First, note that due to size restrictions (i.e. $s_2<s_4$) and Lemma \ref{lem:tripel} we have that $\alpha_2<\alpha_4$ and $\alpha_5<\alpha_3$, i.e. $\alpha_2=\alpha_5<\alpha_3,\alpha_4$. Further size
and divisibility restrictions (Lemma \ref{lem:tripel}) yield Case II in Table \ref{tab:p=3mod4}.

\subsubsection{Case B}
Similarly as in Case A we deduce that one of the three options from Case A holds.
But this time the first option cannot hold, since this would imply $\beta_1=\beta_3=\beta_5$ and therefore $s_3|s_5$, which contradicts Lemma \ref{lem:tripel}. Also the third option yields a contradiction.
Note that under the hypothesis of the third option we have $\beta_3=\beta_4<\beta_1=\beta_5\leq\beta_2<\beta_6$. But since $s_2<s_5$ we have $\alpha_2<\alpha_5$ and since $s_3\nmid s_5$ we have $\alpha_5<\alpha_3$
and Lemma \ref{lem:min_al} applied to the quadruple $(\alpha_2,\alpha_3,\alpha_4,\alpha_5)$ yields $\alpha_2=\alpha_4$, a contradiction to Lemma \ref{lem:tripel}.

Therefore we are left to consider the second option, which yields
$$\beta_1=\beta_4=\beta_5<\beta_2,\beta_3<\beta_6 \quad \text{and}\quad 0=\alpha_1=\alpha_6.$$
Due to size and divisibility restrictions we conclude that $\alpha_2<\alpha_4$ and $\alpha_3<\alpha_5$, thus $\alpha_2=\alpha_3< \alpha_4,\alpha_5$ by an application of Lemma \ref{lem:min_al}.
Further size and divisibility restrictions yield 
$$0=\alpha_1=\alpha_6<\alpha_2=\alpha_3<\alpha_4<\alpha_5\quad \text{and} \quad \beta_1=\beta_4=\beta_5<\beta_2< \beta_3<\beta_6.$$
Let us write $\alpha:=\alpha_2=\alpha_3$ and $\beta:=\beta_1=\beta_4=\beta_5$.
Now we apply Lemma \ref{lem:abcd-div} and deduce that
$$d\left|\frac{s_5-s_3}{\gcd(s_5,s_3)}\right.=\frac{p^{\alpha_5}q^\beta-p^\alpha q^{\beta_3}}{p^{\alpha}q^{\beta}}$$
and therefore we have that $d<p^{\alpha_5-\alpha}$. But, in view of $s_5=bd+1=p^{\alpha_5}q^\beta$ we get $b>p^{\alpha}q^{\beta}$ which is impossible since $b<s_1=ab+1=p^{\alpha}q^{\beta}$.

\subsubsection{Case C}
We have the same options as in Cases A and B respectively. The first option together with the hypotheses of Case C yields
$$\beta_2=\beta_5<\beta_1=\beta_3\leq \beta_4<\beta_6 \quad \text{and}\quad 0=\alpha_1=\alpha_6.$$
Size and divisibility restrictions yield $\alpha_4<\alpha_2$ and $\alpha_3<\alpha_5$, i.e. $\alpha_3=\alpha_4<\alpha_2<\alpha_5$ due to Lemma \ref{lem:min_al}. Thus the first option leads to 
 $$0=\alpha_1=\alpha_6<\alpha_3=\alpha_4<\alpha_2<\alpha_5\quad\text{and}\quad\beta_2=\beta_5<\beta_1=\beta_3\leq \beta_4<\beta_6.$$
We aim to show that this cannot hold. Let us write $\alpha:=\alpha_3=\alpha_4$, $\beta:=\beta_2=\beta_5$ and $\beta':=\beta_1=\beta_3$. Note that due to
Lemma \ref{lem:alpha_beta_zero} we have $\beta>0$. By using the $L$-notation (see Lemma \ref{lem:L-geometric}) and since $s_4>10^5$ we compute
\begin{align*}
 a^2=&\frac{(s_1-1)(s_2-1)}{s_4-1}\\
 =&\frac{s_1s_2}{s_4}+\frac{s_1s_2}{s_4(s_4-1)}-\frac{s_2}{s_4-1}-\frac{s_1}{s_4-1}+\frac{1}{s_4-1}\\
 =&\frac{p^{\alpha_2-\alpha}}{q^{\beta_4-\beta-\beta'}}+L\left(1.001 \cdot \left(\frac{q^{\beta+\beta'}p^{\alpha_2}}{q^{2\beta_4} p^{2\alpha}}+\frac{p^{\alpha_2-\alpha}}{q^{\beta_4-\beta}}+
 \frac{1}{q^{\beta_4-\beta'}p^\alpha}+\frac 1{q^{\beta_4} p^\alpha}\right) \right)
\end{align*}
It is easy to see that each summand in the $L$-term is less than $\frac{1}{5q^{\beta_4-\beta-\beta'}}$ provided that $q^{\beta'}>\frac 15 p^{\alpha_2-\alpha}$. Therefore we obtain
$$a^2=\frac{p^{\alpha_2-\alpha}}{q^{\beta_4-\beta-\beta'}}+L\left(\frac{0.9}{q^{\beta_4-\beta-\beta'}}\right).$$
Thus we conclude that $\frac{p^{\alpha_2-\alpha}}{q^{\beta_4-\beta-\beta'}}$ has to be an integer, if $q^{\beta'}>\frac 15 p^{\alpha_2-\alpha}$. In other words we have shown
that $\beta_4>\beta+\beta'$ implies $p^{\alpha_2-\alpha}\geq 5q^{\beta'}$. However Lemma \ref{lem:s4_low_bound} yields
$$q^{\beta'}\cdot p^\alpha q^\beta =\gcd(s_1,s_4)\gcd(s_2,s_4)<s_4=q^{\beta_4}p^\alpha,$$
hence $\beta_4>\beta+\beta'$. Thus we may deduce that indeed $p^{\alpha_2-\alpha}\geq 5q^{\beta'}$. But considering the inequality $s_2<s_3$ reveals that
$p^{\alpha_2}q^{\beta}<p^{\alpha}q^{\beta'}$ and therefore 
we have the inequality
$$q^{\beta'-\beta}>p^{\alpha_2-\alpha}\geq 5q^{\beta'},$$
which cannot hold.

The second option yields
$$0=\alpha_1=\alpha_6 \quad \text{and}\quad \beta_2=\beta_5<\beta_1=\beta_4\leq \beta_3<\beta_6.$$
By similar size and divisibility restrictions we obtain
$$0=\alpha_1=\alpha_6<\alpha_3=\alpha_4<\alpha_2<\alpha_5 \quad \text{and}\quad \beta_2=\beta_5<\beta_1=\beta_4\leq \beta_3<\beta_6.$$
However we have $\gcd(s_1,s_4)= p^{\alpha_1}q^{\beta_4}$ and $\gcd(s_2,s_4)=p^{\alpha_4} q^{\beta_2}$ which yields in view of Lemma \ref{lem:s4_low_bound} the inequality
$$p^{\alpha_1}q^{\beta_4}\cdot p^{\alpha_4}q^{\beta_2}=\gcd(s_1,s_4)\gcd(s_2,s_4) <s_4=p^{\alpha_4}q^{\beta_4}.$$
Thus the second option cannot hold.

If we consider the third option we obtain the following two possibilities
$$0=\alpha_1=\alpha_6 \quad \text{and}\quad \beta_2=\beta_5<\beta_3=\beta_4\leq \beta_1<\beta_6 $$
or
$$0=\alpha_1=\alpha_6 \quad \text{and}\quad \beta_3=\beta_4<\beta_2=\beta_5\leq \beta_1<\beta_6$$
respectively. By applying Lemma \ref{lem:min_al} to the quadruple $(\alpha_2,\alpha_3,\alpha_4,\alpha_5)$ and taking size and divisibility considerations into account we obtain
$$0=\alpha_1=\alpha_6<\alpha_3=\alpha_4<\alpha_2<\alpha_5 \quad \text{and}\quad \beta_2=\beta_5<\beta_3=\beta_4\leq \beta_1<\beta_6 $$
or
$$0=\alpha_1=\alpha_6<\alpha_2=\alpha_5<\alpha_3,\alpha_4 \quad \text{and}\quad \beta_3=\beta_4<\beta_2=\beta_5\leq \beta_1<\beta_6$$
respectively. But both possibilities yield a contradiction, if we apply Lemma \ref{lem:s4_low_bound}. Indeed we obtain
$$q^{\beta_4} \cdot p^{\alpha_4} q^{\beta_2}=\gcd(s_1,s_4)\gcd(s_2,s_4)<s_4=p^{\alpha_4} q^{\beta_4}$$
or
$$q^{\beta_5} \cdot p^{\alpha_5} q^{\beta_3}=\gcd(s_1,s_5)\gcd(s_3,s_5)<s_5=p^{\alpha_5} q^{\beta_5} $$
respectively.

\subsection{The case that $\alpha_2=\alpha_5=0$}

In this case we have that $\beta_2<\beta_5$ and by applying Lemma \ref{lem:min_al} to the quadruple $(\beta_2,\beta_3,\beta_4,\beta_5)$ we obtain that $\beta_3=\beta_4<\beta_2<\beta_5$. Indeed as noted
$\beta_5$ cannot be minimal, but also $\beta_2$ cannot be minimal since $0=\alpha_2<\alpha_4$ and we would obtain that $s_2|s_4$, a contradict to Lemma \ref{lem:tripel}. Therefore we obtain
$$0=\alpha_2=\alpha_5 \quad \text{and} \quad \beta_3=\beta_4<\beta_2<\beta_5.$$
Now we apply Lemma \ref{lem:min_al} to the quadruple $(\beta_1,\beta_2,\beta_5,\beta_6)$. But $\beta_5$ cannot be minimal since otherwise we would obtain $s_5|s_6$ contradicting Lemma \ref{lem:tripel}
and $\beta_2$ cannot be minimal since otherwise we would obtain $s_2<s_1$ also an obvious contradiction. Since neither $\beta_2$ nor $\beta_5$ can be minimal we have
$\beta_1=\beta_6<\beta_2<\beta_5$. Once again we apply Lemma \ref{lem:min_al}, this time to the quadruple $(\alpha_1,\alpha_3,\alpha_4,\alpha_6)$. Since $\beta_1=\beta_6$ we conclude that $\alpha_1<\alpha_6$ and
that $\alpha_6$ cannot be minimal. This leaves us with the following three options:
\begin{itemize}
 \item $\alpha_1=\alpha_3\leq \alpha_4,\alpha_6$;
 \item $\alpha_1=\alpha_4\leq \alpha_3,\alpha_6$;
 \item $\alpha_3=\alpha_4\leq \alpha_1<\alpha_6$.
\end{itemize}
The first option combined with the previous found restriction
$$ 0=\alpha_2=\alpha_5 \quad \text{and}\quad \beta_3=\beta_4,\beta_1=\beta_6<\beta_2<\beta_5 $$
yields Case III after taking further size and divisibility restrictions into account.

The second option yields in view of our usual size and divisibility restrictions 
$$0=\alpha_2=\alpha_5<\alpha_1=\alpha_4\leq \alpha_3< \alpha_6 \quad \text{and} \quad \beta_1=\beta_6<\beta_3=\beta_4<\beta_2<\beta_5. $$
However this yields
$$p^{\alpha_4} q^{\beta_1} \cdot p^{\alpha_2}q^{\beta_4} =\gcd(s_4,s_1)\gcd(s_4,s_2) > s_4=p^{\alpha_4}q^{\beta_4}$$
which contradicts Lemma~\ref{lem:s4_low_bound}.

For the third option we deuce that
$$0=\alpha_2=\alpha_5<\alpha_3=\alpha_4< \alpha_1< \alpha_6 \quad \text{and}\quad \beta_1=\beta_6<\beta_3=\beta_4<\beta_2<\beta_5 $$
by size and divisibility restrictions. But this yields a contradiction since by Lemma~\ref{lem:s4_low_bound} we have
$$p^{\alpha_4} q^{\beta_1} \cdot p^{\alpha_2}q^{\beta_4} =\gcd(s_4,s_1)\gcd(s_4,s_2)<s_4=p^{\alpha_4}q^{\beta_4}. $$

\subsection{The case that $\alpha_3=\alpha_4=0$}

We consider the quadruple $(\beta_2,\beta_3,\beta_4,\beta_5)$ and note that $\beta_3$ cannot be minimal since otherwise $s_3|s_5$ (a contradiction to Lemma \ref{lem:tripel}) and
that $\beta_4$ cannot be minimal since otherwise $s_4<s_2$, again a contradiction. Thus we have
$$0=\alpha_3=\alpha_4 \quad \text{and} \quad \beta_2=\beta_5<\beta_3,\beta_4.$$

Next, we consider the quadruple $(\beta_1,\beta_3,\beta_4,\beta_6)$ and similar as above we have that $\beta_3$ cannot be minimal since otherwise $s_3|s_6$ and $\beta_4$ cannot be minimal since otherwise
$s_4<s_1$. Thus we have $\beta_1=\beta_6<\beta_3,\beta_4$.

Finally, we consider the quadruple $(\alpha_1,\alpha_2,\alpha_5,\alpha_6)$ and obtain that $\alpha_2<\alpha_5$ since $\beta_2=\beta_5$ and $\alpha_1<\alpha_6$ since $\beta_1=\beta_6$. Therefore
Lemma \ref{lem:min_al} yields $\alpha_1=\alpha_2<\alpha_5,\alpha_6$. Putting all pieces together we obtain Case IV in Table \ref{tab:p=3mod4}.

We have chased down all possible cases and found no other possibilities for the exponents than those described in Table \ref{tab:p=3mod4}. Therefore Proposition \ref{prop:p=3mod4} is proved.

\subsection{The case that $p=2$}

By similar arguments we can also deal with the case that $p=2$. In particular we obtain the following proposition:

\begin{proposition}\label{prop:p=2}
  Let $q>2$ be a prime and assume that a $\{2,q\}$-Diophantine quadruple exists. Then one of the four cases in Table \ref{tab:p=2} holds.
\end{proposition}

\begin{table}[ht]
\caption{Restrictions to the exponents}\label{tab:p=2}
 \begin{tabular}{|c|l|l|}
 \hline Case & The $\alpha$ exponents & The $\beta$ exponents \\ \hline\hline
 I & $1=\alpha_1=\alpha_6<\alpha_4=\alpha_5<\alpha_2<\alpha_3$ & $\beta_1=\beta_2=\beta_3<\beta_4<\beta_5<\beta_6$ \\\hline
 II & $1=\alpha_1=\alpha_6<\alpha_2=\alpha_5<\alpha_3,\alpha_4$ & $\beta_3=\beta_4<\beta_1=\beta_2< \beta_5<\beta_6$ \\\hline
 III & $1=\alpha_2=\alpha_5<\alpha_1=\alpha_3\leq \alpha_4<\alpha_6$ & $\beta_1=\beta_6<\beta_3=\beta_4<\beta_2<\beta_5$ \\\hline
 IV & $1=\alpha_3=\alpha_4<\alpha_1=\alpha_2<\alpha_5<\alpha_6$ & $\beta_1=\beta_6<\beta_2=\beta_5< \beta_3,\beta_4$\\ \hline
 \end{tabular}

\end{table}

\begin{proof}
 By \cite[Lemma 2.4]{Szalay:2015} we know that, if $(a,b,c,d)$ is a $\{2,q\}$-Diophantine quadruple, then $a,b,c,d$ are all odd. Thus all $S$-units $s_1,\dots,s_6$ are even, hence
 $1\leq \alpha_1,\dots,\alpha_6$. Now the same chase of cases yields the same result, but with $0$ replaced by $1$.
\end{proof}

\section{The $p$-adic Wieferich condition}\label{Sec:Wief_padic}

In this section we want to investigate the case that $v_p(q^{p-1}-1)< \max\left\{2,\frac {\log q}{\log p}\right\}$. In view of our main Theorem \ref{th:gen} we would like 
to prove that under this $p$-adic divisibility assumption no $\{p,q\}$-Diophantine quadruple exists. In this section we prove such a result provided that $q$ is large enough.  

\begin{proposition}\label{prop:Wief_padic}
Let $p<q$ be odd primes such that $p\equiv 3 \mod 4$ and assume that $v_p(q^{p-1}-1)< \max\left\{2,\frac {\log q}{\log p} \right\}$. If there exists a $\{p,q\}$-Diophantine quadruple, then $q\leq 700393$. 
To be more precise, if there exists a  $\{p,q\}$-Diophantine quadruple, then
$$p<q<52038\log p.$$
\end{proposition}

In order to prove Proposition \ref{prop:Wief_padic} we have to consider each case of Table \ref{tab:p=3mod4} individually. However, before we start with the proof of Proposition \ref{prop:Wief_padic}
let us state the following lemma which will be frequently used throughout this section.

\begin{lemma}\label{lem:p-adic-bound}
 Let $p<q$ be odd primes such that $p\equiv 3 \mod 4$ and let $*,\dagger \in \{0,1,2,3,4,5,6\}$. Moreover, set $\beta_0=0$ and assume that $q\geq 52038 \log p$.
 If $v_p(q^{p-1}-1)< \max\left\{2,\frac {\log q}{\log p} \right\}$, then we have 
 $$p^{v_p\left(q^{|\beta_*-\beta_\dagger|}\pm 1\right)}<q^2.$$
\end{lemma}

\begin{proof}
 Since Proposition \ref{prop:bound_AB} we know that $|\beta_*-\beta_\dagger|<52038\log p$. Due to Lemma \ref{lem:p-adic} and our assumption that $u_p<\max\left\{2,\frac {\log q}{2\log p} \right\}$ we obtain
 \begin{align*}
  v_p\left(q^{|\beta_*-\beta_\dagger|}-1\right)=& u_p+v_p(|\beta_*-\beta_\dagger|)\\
  <&\max\left\{2,\frac {\log q}{\log p} \right\} +\frac{\log (52038 \log p)}{\log p}\\
  \leq&\frac{2\log q}{\log p},
 \end{align*}
 provided that $52038 \log p\leq q$, which implies immediately the statement of the lemma.
\end{proof}

Let us assume that Proposition \ref{prop:Wief_padic} is false and that $q\geq 52038 \log p$, i.e. Lemma \ref{lem:p-adic-bound} applies and $p^{v_p\left(q^{|\beta_*-\beta_\dagger|}\pm 1\right)}<q^2$.
We will show that under this assumption we obtain a contradiction in each case described in Proposition \ref{prop:p=3mod4} (see Table \ref{tab:p=3mod4}).

\subsection{Case I}

In this case we have 
\begin{align*}
 ab+1=&q^{\beta}, &  bc+1=&p^{\alpha}q^{\beta_4}, \\
 ac+1=&p^{\alpha_2}q^{\beta}, &  bd+1=&p^{\alpha}q^{\beta_5}, \\
 ad+1=&p^{\alpha_3}q^{\beta}, &  cd+1=&q^{\beta_6},
\end{align*}
with $0<\alpha<\alpha_2<\alpha_5$ and $\beta<\beta_4<\beta_5<\beta_6$. Note that we may assume that $\beta>0$ since otherwise $s_1=ab+1=1$ would yield an obvious contradiction.
We consider the second equation of system \eqref{eq:SUnit} and obtain after dividing through the common denominator $p^\alpha q^\beta$ the
equation
\begin{equation}\label{eq:CaseI}
 q^{\beta_4-\beta}\left(q^{\beta_5-\beta_4}-1\right)=p^{\alpha_2}q^{\beta_5}-p^{\alpha_3}q^{\beta_4}+p^{\alpha_3-\alpha}-p^{\alpha_2-\alpha}.
\end{equation}
That is $p^{\alpha_2-\alpha}|q^{\beta_5-\beta_4}-1$ and Lemma \ref{lem:p-adic-bound} implies that $p^{\alpha_2-\alpha}<q^2$.

Next, we observe that since Lemma \ref{lem:s4_low_bound} we have 
$$q^\beta \cdot p^{\alpha}q^\beta=\gcd(s_1,s_4)\gcd(s_2,s_4)<s_4=p^{\alpha}q^{\beta_4},$$
i.e. $2\beta< \beta_4$. On the other hand we have that $(ab+1)(ac+1)>bc+1$ and obtain
$$p^{\alpha_2}q^{2\beta}>p^\alpha q^{\beta_4}.$$
Therefore we obtain that $2\beta+2>\beta_4$, since $p^{\alpha_2-\alpha}<q^2$, hence $\beta_4=2\beta+1$.

Next we compute $a^2$. Using the $L$-notation and Lemma \ref{lem:L-geometric} we obtain:
\begin{align*}
 a^2 =& \frac{(s_1-1)(s_2-1)}{s_4-1}\\
 =&\frac{s_1s_2}{s_4}+\frac{s_1s_2}{s_4(s_4-1)}-\frac{s_2}{s_4-1}- \frac{s_1}{s_4-1}+\frac{1}{s_4-1}  \\
 =&\frac{p^{\alpha_2-\alpha}}{q}+\frac{p^{\alpha_2-\alpha}}{q(p^\alpha q^{2\beta+1}-1)}-\frac{p^{\alpha_2-\alpha}}{q^{\beta+1}-\frac{1}{q^\beta}}-\frac{1}{p^\alpha q^{\beta+1}-\frac 1{q^\beta}}+\frac{1}{p^\alpha q^{2\beta+1}-1}\\
 =& \frac{p^{\alpha_2-\alpha}}{q}+L\left(1.001\cdot \left(\frac{p^{\alpha_2-\alpha}}{p^{\alpha}q^{2\beta+2}}+\frac{p^{\alpha_2-\alpha}}{q^{\beta+1}}
+\frac{1}{p^\alpha q^{\beta+1}}+\frac{1}{p^\alpha q^{2\beta+1}}\right)\right)\\
 =& \frac{p^{\alpha_2-\alpha}}{q}+L\left(1.001\cdot \left(\frac{1}{p^{\alpha}q^{2\beta}}+\frac{1}{q^{\beta-1}}
 +\frac{1}{p^\alpha q^{\beta+1}}+\frac{1}{p^\alpha q^{2\beta+1}}\right)\right)
\end{align*}
To obtain the last inequality we used the fact that $p^{\alpha_2-\alpha}<q^2$. Let us assume for the moment that $\beta\geq 3$. Then we obtain
$$a^2=\frac{p^{\alpha_2-\alpha}}{q}+L\left(\frac{4.004}{q^2}\right).$$
Obviously $\frac{p^{\alpha_2-\alpha}}{q}$ is not an integer. But the distance to the nearest integer is at least $\frac 1q$. Thus we deduce that $a^2$ cannot be an integer, hence a contradiction.

Therefore we may assume that $0<\beta\leq 2$. Let us discuss the case that $\beta=1$ first. A similar computation of $a^2$ as before shows that
$$a^2=\frac{p^{\alpha_2-\alpha}}{q}-\frac{p^{\alpha_2-\alpha}}{q^2}+L\left(1.001 \cdot \left(\frac{1}{p^{\alpha}q^{2}}+\frac{1}{p^{\alpha} q^{3}}
 +\frac{1}{p^{\alpha} q^{2}}+\frac{1}{p^\alpha q^{3}} \right)\right)
 $$
and since $p^\alpha\geq p\geq 3$ and $q>10^5$ we conclude that 
$$a^2=\frac{p^{\alpha_2-\alpha}}{q}-\frac{p^{\alpha_2-\alpha}}{q^2}+L\left(\frac{0.7}{q^2}\right)$$
is an integer. This implies that
$$q^2\left| \left(qp^{\alpha_2}-p^{\alpha_2}\right)\right.$$
which is an obvious contradiction.

In the case that $\beta=2$ we deduce by reconsidering the asymptotic expansion of $a^2$ a similar contradiction.

\subsection{Case II}

In the second case we have 
\begin{align*}
 ab+1=&q^{\beta'}, &  bc+1=&p^{\alpha_4}q^{\beta}, \\
 ac+1=&p^{\alpha}q^{\beta'}, &  bd+1=&p^{\alpha}q^{\beta_5}, \\
 ad+1=&p^{\alpha_3}q^{\beta}, &  cd+1=&q^{\beta_6},
\end{align*}
with $0<\alpha<\alpha_3,\alpha_4$ and $\beta<\beta'<\beta_5<\beta_6$. Note that $\beta=0$ can be excluded due to Lemma \ref{lem:alpha_beta_zero}. We consider the second equation of system \eqref{eq:SUnit} and obtain 
\begin{equation}\label{eq:CaseII}
 q^{\beta'-\beta}\left(q^{\beta_5-\beta'}+1\right)=p^{\alpha}q^{\beta_5+\beta'-\beta}+p^{\alpha_3-\alpha}+p^{\alpha_4-\alpha}-p^{\alpha_3+\alpha_4-\alpha}q^\beta.
\end{equation}
Let $M=\min\{\alpha,\alpha_3-\alpha,\alpha_4-\alpha\}$, then $p^{M}|q^{\beta_5-\beta'}+1$ and Lemma \ref{lem:p-adic-bound} implies that $p^{M}<q^2$.
That is we have to consider three subcases:
\begin{description}
 \item[Case A] $M=\alpha$ or
 \item[Case B] $M=\alpha_3-\alpha$ or
 \item[Case C] $p^{\alpha_4-\alpha}<q^2$ and $M\neq \alpha$.
\end{description}
Let us note that we will show that Case B will imply that $p^{\alpha_4-\alpha}<q^2$. This is the reason why we consider this more general statement for Case C.

\subsubsection{Case A}

First, we observe that
$$
\frac{c}{b}=\frac{s_2-1}{s_1-1}=\frac{q^{\beta'}p^\alpha-1}{q^{\beta'}-1}=p^\alpha \frac{q^{\beta'}}{q^{\beta'}-1}-\frac1{q^{\beta'}-1}<1.001 p^{\alpha}
$$
since $q^{\beta'}\geq q>10^5$. Therefore we have
$$q^{\beta_6}=cd+1<\frac{c}{b}(bd+1)<1.001p^{2\alpha}q^{\beta_5}$$
and we obtain that $\beta_6-\beta_5\leq 4$.

Let us assume that $\beta_6-\beta_5=3$. Then Lemma \ref{lem:abcd-div} yields that $d|\frac{s_6-s_5}{\gcd(s_6,s_5)}=\frac{q^{\beta_6}-q^{\beta_5}p^\alpha}{q^{\beta_5}}$
and we obtain that $d<q^3$. Further, we deduce that $q^6>d^2>cd+1=q^{\beta_6}$. Therefore $\beta_6\leq 5$ and by our assumption that $\beta_6-\beta_5=3$  we obtain that
$\beta_5\leq 2$. On the other hand we have $0<\beta<\beta'<\beta_5\leq 2$ which yields an obvious contradiction. Similar arguments show that $\beta_6-\beta_5=1,2$ cannot hold.

Therefore we are left with the case that $\beta_6-\beta_5=4$. In this case Lemma \ref{lem:abcd-div} yields similarly as above that $\beta_6\leq 7$ and we obtain
$$\beta=1,\qquad \beta'=2,\qquad \beta_5=3,\qquad \beta_6=7.$$
Let us apply Lemma \ref{lem:s4_low_bound}, then we obtain that
$$q^3\cdot p^\alpha=\gcd(s_1,s_5)\gcd(s_3,s_5)<s_5=q^3p^{\alpha}$$
which is a contradiction.

\subsubsection{Case B}

Due to Case A we may assume that $\alpha>\alpha_3-\alpha$.
If we consider the inequality $s_2<s_3$, then we obtain that $q\leq q^{\beta'-\beta}<p^{\alpha_3-\alpha}<q^2$ and therefore $\beta'=\beta+1$. Note that this inequality also
implies that $q<p^{\alpha_3-\alpha}$ and therefore we also have that $p^\alpha>q$ and $p^{\alpha_4}=p^{\alpha_4-\alpha}p^\alpha>q^2$. Next, we compute
\begin{align*}
 a^2=& \frac{(s_1-1)(s_2-1)}{s_4-1}\\
 =& \frac{p^{\alpha}q^{2\beta'}-p^{\alpha}q^{\beta'}-q^{\beta'}+1}{p^{\alpha_4} q^{\beta}-1}\\
 =&  \left(p^{\alpha}q^{2\beta'}-p^{\alpha}q^{\beta'}\right)\left(\frac{1}{p^{\alpha_4} q^{\beta}}+\frac{1}{p^{\alpha_4} q^{\beta}(p^{\alpha_4} q^{\beta}-1)}\right)-
 (q^{\beta'}-1)\frac{1}{p^{\alpha_4} q^{\beta}-1}\\
 =& \frac{q^{\beta'+1}-q}{p^{\alpha_4-\alpha}}+\stackrel{X:=}{\overbrace{ \frac{q^{\beta'+1}-q}{p^{\alpha_4-\alpha}(p^{\alpha_4} q^{\beta}-1)}}}
 -\stackrel{Y:=}{\overbrace{ (q^{\beta'}-1)\frac{1}{p^{\alpha_4} q^{\beta}-1}}}
\end{align*}
We obviously have $X,Y>0$. We want to show that also $X,Y<p^{\alpha-\alpha_4}$ holds. Indeed $X<p^{\alpha-\alpha_4}$ holds since
$$X<\frac{q^{\beta'+1}}{p^{\alpha_4-\alpha} (p^{\alpha_4} q^{\beta}-1)}=\frac{1}{p^{\alpha_4-\alpha}}\cdot \frac{q^2}{p^{\alpha_4}-\frac{1}{q^\beta}}
<\frac{1}{p^{\alpha_4-\alpha}}.$$
In particular note that $p^{\alpha_4}-\frac{1}{q^\beta}>q^2$ since $p^{\alpha_4}>q^2$ and both $p^{\alpha_4}$ and $q^2$ are integers.
A similar computation shows that $Y<p^{\alpha-\alpha_4}$ holds:
$$Y<\frac{q^{\beta'}}{p^{\alpha_4} q^{\beta}-1}=\frac{1}{p^{\alpha_4-\alpha}}\cdot \frac{q}{p^{\alpha}-\frac{1}{q^\beta p^{\alpha_4-\alpha}}}
<\frac{1}{p^{\alpha_4-\alpha}}.$$
Therefore $|X-Y|<\frac{1}{p^{\alpha_4-\alpha}}$ and  $\frac{q^{\beta'+1}-q}{p^{\alpha_4-\alpha}}$ has to be an integer. Thus
$$p^{\alpha_4-\alpha}|q^{\beta'}-1$$
and Lemma \ref{lem:p-adic-bound} yields $p^{\alpha_4-\alpha}<q^2$. But this implies that we are in Case C.

\subsubsection{Case C}

First of all note that we may assume that $p^\alpha>q$, since $p^\alpha<q$ would imply that either $p^{\alpha_3-\alpha}<q$ or that $p^{\alpha_4-\alpha}<q$. The first case can be
excluded as in Case B. The second case can be excluded by using Lemma \ref{lem:s4_low_bound}:
$$q^\beta\cdot p^\alpha q^\beta= \gcd(s_4,s_1)\gcd(s_4,s_2)<s_4=q^{\beta}p^{\alpha_4}$$
which yields $q^\beta<p^{\alpha_4-\alpha}$. If $p^{\alpha_4-\alpha}<q$ we would deduce that $\beta=0$ a contradiction to Lemma \ref{lem:alpha_beta_zero}. Nevertheless we
obtain $\beta=1$. By the inequality $s_2<s_4$ we also obtain that $q^{\beta'-\beta}<p^{\alpha_4-\alpha}<q^2$, hence $\beta=1$ and $\beta'=2$.

Next we observe that by Lemma \ref{lem:abcd-div} we have $c|\frac{s_4-s_2}{\gcd(s_4,s_2)}$ and therefore $c<\frac{qp^{\alpha_4}}{qp^{\alpha}}<q^2$. Since
$ac+1=q^2p^\alpha$ this yields that $a>p^\alpha>q$ and therefore we get $q^2=ab+1>p^{2\alpha}>q^2$, a contradiction. 

\subsection{Case III}

In the third case we have:
\begin{align*}
 ab+1=&p^\alpha q^{\beta}, &  bc+1=&p^{\alpha_4}q^{\beta'}, \\
 ac+1=&q^{\beta_2}, &  bd+1=&q^{\beta_5}, \\
 ad+1=&p^{\alpha}q^{\beta'}, &  cd+1=&p^{\alpha_6}q^{\beta},
\end{align*}
with $0<\alpha\leq\alpha_4<\alpha_6$ and $\beta<\beta'<\beta_2<\beta_5$.
We consider the last equation of system \eqref{eq:SUnit} and obtain in combination with Lemma \ref{lem:p-adic-bound} that $p^{\alpha_4-\alpha}<q^2$. Indeed, we can do slightly better
and show that $p^{\alpha_4-\alpha}<q^2-1$, since $\gcd(q+1,q-1)=2$ and $p^{\alpha_4-\alpha}=q^2-1=(q-1)(q+1)$ would yield a contradiction. We use Lemma \ref{lem:s4_low_bound}
to obtain
$$q^{\beta}p^\alpha \cdot q^{\beta'}=\gcd(s_4,s_1)\gcd(s_4,s_2)<s_4=q^{\beta'}p^{\alpha_4}, $$
and therefore $q^\beta<p^{\alpha_4-\alpha}<q^2$, which implies that $\beta=1$. Also note that since $\beta=0$ is excluded due to Lemma \ref{lem:alpha_beta_zero} we  
conclude that $q<p^{\alpha_4-\alpha}$. We also note that due to $s_2<s_3$ we have that $p^\alpha>q$, hence $p^{\alpha_4}>q^2$. 

Now, let us consider the second equation of system \eqref{eq:SUnit}. After dividing through a common denominator and rearranging terms we obtain
$$p^\alpha(p^{\alpha_4-\alpha}+1)=q^{\beta'}p^{\alpha_4}-q^{\beta_2+\beta_5-\beta'}+q^{\beta_2-\beta'}+q^{\beta_5-\beta'}.$$
Since $p^{\alpha_4-\alpha}+1<q^2$ the $q$-adic valuation of the right hand side is at most $1$. Therefore we have either $\beta'=1$ or $\beta_2-\beta'=1$. Since $\beta=\beta'=1$
is a contradiction we have that $\beta_2=\beta'+1$. 

Next, we compute
\begin{align*}
 a^2=&\frac{(s_1-1)(s_2-1)}{s_4-1}\\
 =& \frac{p^{\alpha}q^{\beta'+2}-q^{\beta'+1}-p^{\alpha}q+1}{p^{\alpha_4}q^{\beta'}-1}\\
 =&  \frac{q^2}{p^{\alpha_4-\alpha}}-\frac{q}{p^{\alpha_4}}+L\left(1.001 \cdot \left(\frac{q^2}{p^{2\alpha_4-\alpha}q^{\beta'}}+ \frac{q}{p^{2\alpha_4}q^{\beta'}}+\frac{q}{p^{\alpha_4-\alpha}q^{\beta'}}
 +\frac{1}{p^{\alpha_4} q^{\beta'}}\right)\right)\\ 
 =&  \frac{q^2}{p^{\alpha_4-\alpha}}-\frac{q}{p^{\alpha_4}}+L\left(\frac{1.001}{p^{\alpha_4-\alpha}} \left(\frac{1}{q^2}+ \frac{1}{q^5}+\frac{1}{q^2}+\frac{1}{q^4}\right)\right)\\
 =&  \frac{q^2}{p^{\alpha_4-\alpha}}-\frac{1}{p^{\alpha_4-\alpha}}\cdot \frac{q}{p^\alpha}+L\left(\frac{2.004}{p^{\alpha_4-\alpha}}\cdot \frac{1}{q}\right).
\end{align*}
Therefore we have that $1>a^2p^{\alpha_4-\alpha}-q^2>-2$. This implies that
$$a^2=\frac{q^2-1}{p^{\alpha_4-\alpha}}\qquad \text{or} \qquad a^2=\frac{q^2}{p^{\alpha_4-\alpha}}$$
is an integer. The second option clearly cannot hold. Thus $p^{\alpha_4-\alpha}|q^2-1=(q+1)(q-1)$. 
Since $\gcd(q+1,q-1)=2$ and $p$ is odd we conclude that $p^{\alpha_4-\alpha}\leq \frac{q+1}2<q$ which contradicts our previous
conclusion that $q<p^{\alpha_4-\alpha}$.

\subsection{Case IV}

In the last case we have:
\begin{align*}
 ab+1=&p^\alpha q^{\beta}, &  bc+1=&q^{\beta_4}, \\
 ac+1=&p^{\alpha}q^{\beta'}, &  bd+1=&p^{\alpha_5}q^{\beta'}, \\
 ad+1=&q^{\beta_3}, &  cd+1=&p^{\alpha_6}q^{\beta},
\end{align*}
with $0<\alpha<\alpha_5<\alpha_6$ and $0<\beta<\beta'<\beta_3,\beta_4$. Note that $\beta=0$ is not possible due to Lemma \ref{lem:alpha_beta_zero}
In this case we consider the first equation of system~\eqref{eq:SUnit} and obtain in combination with Lemma~\ref{lem:p-adic-bound} that $p^{\alpha_5-\alpha}<q^2$. Similar as in Case III we may even
assume that $p^{\alpha_5-\alpha}<q^2-1$. Now applying Lemma \ref{lem:s4_low_bound}
we obtain
$$p^\alpha q^\beta \cdot q^{\beta'}=\gcd(s_5,s_1)\gcd(s_5,s_3)<s_5=p^{\alpha_5}q^{\beta'},$$
hence $q^\beta<p^{\alpha_5-\alpha}<q^2$ and therefore $\beta=1$. Now, let us consider the second equation of system~\eqref{eq:SUnit}. After dividing through a common denominator and rearranging terms we obtain
$$p^\alpha(p^{\alpha_5-\alpha}+1)=q^{\beta'}p^{\alpha_5+\alpha}-q^{\beta_3+\beta_4-\beta'}+q^{\beta_3-\beta'}+q^{\beta_4-\beta'}.$$
Since $p^{\alpha_5-\alpha}+1<q^2$ the $q$-adic valuation of the right hand side is at most $1$. Therefore we have either $\beta'=1$ or $\beta_4-\beta'=1$ or $\beta_3-\beta'=1$. Since $\beta=\beta'=1$
is a contradiction the first case cannot hold.

Let us assume that $\beta_4=\beta'+1$. We apply Lemma \ref{lem:s4_low_bound} and obtain
$$q\cdot q^{\beta'}=\gcd(s_4,s_1)\gcd(s_4,s_2)<s_4=q^{\beta'+1}, $$
an obvious contradiction.

Therefore we have $\beta_3=\beta'+1$ and $\beta_4\geq \beta'+2$. But $\beta_3=\beta'+1$ yields together with $s_2<s_3$ that $p^\alpha<q$ and therefore $p^{\alpha_5}<q^3$ and due to $s_4<s_5$ we
deduce that $\beta_4=\beta'+2$. Let us compute
\begin{align*}
 a^2=&\frac{(s_1-1)(s_2-1)}{s_4-1}\\
 =& \frac{p^{2\alpha}q^{1+\beta'}-p^{\alpha}q^{\beta'}-p^{\alpha}q+1}{q^{\beta'+2}-1}\\
 =&  \left(p^{2\alpha}q^{1+\beta'}-p^{\alpha}q^{\beta'}\right)\left(\frac{1}{q^{\beta'+2}}+\frac{1}{q^{\beta'+2}(q^{\beta'+2}-1)}\right)-
 (p^\alpha q-1)\frac{1}{q^{\beta'+2}-1}\\
 =& \frac{p^{\alpha}(p^{\alpha}q-1)}{q^2}+\stackrel{X:=}{\overbrace{ \frac{p^\alpha(p^\alpha q-1)}{q^2(q^{\beta'+2}-1)}}}
 -\stackrel{Y:=}{\overbrace{ (p^\alpha q-1)\frac{1}{q^{\beta'+2}-1}}}
\end{align*}
Since $p^\alpha<q$ it is easy to see that $0<X,Y<\frac{1}{q^2}$ and therefore $|X-Y|<\frac{1}{q^2}$. Hence $\frac{p^{\alpha}(p^{\alpha}q-1)}{q^2}$ is an 
integer, which contradicts the fact that $q^2\nmid p^{\alpha}(p^{\alpha}q-1)$.

Since in all four cases which are described in Proposition \ref{prop:p=3mod4} (see Table \ref{tab:p=3mod4}) the assumption that $q\geq 52038 \log p$ yields a contradiction we have
proved Proposition \ref{prop:Wief_padic} completely.

\section{The case $p=2,3$}\label{Sec:p=3}

We start with the easier case that $p=3$. The first main result of this section is the following proposition which is the content of the first part of Theorem \ref{th:p23}:

\begin{proposition}\label{prop:Wief_3}
There is no $\{3,q\}$-Diophantine quadruple.   
\end{proposition}

\begin{proof}
 We start by noting that $u_3=v_3(q^{\ord_3(q)}-1)\leq v_3(q^2-1)$. Since $q^2-1=(q-1)(q+1)$ and $3\nmid \gcd(q+1,q-1)=2$ we deduce that $3^{u_3}\leq \frac{q+1}2$,
 i.e. $u_3<\max\left\{2,\frac{\log q}{\log 3}\right\}$ and the $p$-adic Wieferich condition is fulfilled and we may apply Proposition \ref{prop:Wief_padic} and deduce that
 $52038 \log p<57170$, if $p=3$. Therefore Proposition \ref{prop:Wief_padic} implies that there is no $\{3,q\}$-Diophantine quadruple
 if $q$ is a prime $\geq 57170$. However, by Lemma \ref{lem:max_pq} we also know that no $\{p,q\}$-Diophantine quadruple exists, if $\max\{p,q\}<10^5$. Thus no $\{3,q\}$-Diophantine
 quadruple exists.
\end{proof}

Now we turn to the much more difficult proof that no $\{2,q\}$-Diophantine quadruple exists. First, let us note that due to \cite[Theorem 1.2]{Szalay:2015} we may assume that
$q\equiv 1 \mod 4$. However we start with an analog of Lemma \ref{lem:p-adic} for the case that $p=2$ and $q\equiv 1 \mod 4$.

\begin{lemma}\label{lem:2-adic}
 Let $q\equiv 1 \mod 4$ be a prime. Then we have
 $$v_2(q^x-1)= u_2+v_2(x)$$
 and
 $$v_2(q^x+1)=1$$
\end{lemma}

\begin{proof}
The second statement is almost trivial after noting that for a prime $q\equiv 1\mod 4$ we have $q^x\equiv 1 \mod 4$.

 In order to prove the first statement of the lemma, assume that $v_2(q^x-1)=\ell$. Then 
 $$q^x\equiv 1+k \cdot 2^\ell \mod 2^{\ell+2}$$
 for some odd integer $k$ and we have
 $$q^{2x}\equiv 1+k\cdot 2^{\ell+1} \mod 2^{\ell+2},$$
 i.e. $v_2(q^{2x}-1)=\ell+1$. For some odd integer $u$ we obtain
 $$q^{ux}\equiv 1+uk\cdot 2^{\ell} \mod 2^{\ell+2},$$
 i.e. $v_2(q^{ux}-1)=\ell$.
 Note that since $q\equiv 1 \mod 4$ we have that $\ell\geq 2$. An induction argument similar as in the case that $p$ is odd can be applied (e.g. see \cite[Section 2.1.4]{Cohen:NTI}).
\end{proof}

In view of Lemma \ref{lem:2-adic} we obtain now the following variant of Lemma \ref{lem:p-adic-bound}:

\begin{lemma}\label{lem:2-adic-bound}
 Let $*,\dagger \in \{0,1,2,3,4,5,6\}$ and set $\beta_0=0$ and $\alpha_0=0$.
 Then we have that
 $$2^{v_2\left(q^{|\beta_*-\beta_\dagger|}-1\right)}< 2^{15} q<q^{1.51}$$
 and
 $$v_2\left(q^{|\beta_*-\beta_\dagger|}+1\right)=1.$$
 
 Moreover we have
 $$v_q\left(2^{|\alpha_*-\alpha_\dagger|}\pm 1\right)=0,u_q.$$
\end{lemma}

\begin{proof}
Note that due to the results of Szalay and the author \cite{Szalay:2015} we may assume that $q\equiv 1 \mod 4$ and $q>10^9$. The statement that
$v_2\left(q^{|\beta_*-\beta_\dagger|}+1\right)=1$ is a direct consequence of Lemma~\ref{lem:2-adic}.

Since $u_2=v_2(q-1)$ we deduce that $u_2<\frac{\log q}{\log 2}$, hence Lemma \ref{lem:2-adic} implies
$$\exp\left(\log 2 \cdot v_2\left(q^{|\beta_*-\beta_\dagger|}-1\right)\right)<\exp(\log q+\log 2 \cdot v_2(|\beta_*-\beta_\dagger|))=q\cdot 2^{v_2(|\beta_*-\beta_\dagger|)}.$$
Since Corollary \ref{cor:bound_23} we know that $|\beta_*-\beta_\dagger|<36070$. Therefore we have that $v_2(|\beta_*-\beta_\dagger|)\leq 15$ and we obtain the first
statement of the lemma, if we take into account that we may assume that $q>10^9$.

The last statement of the lemma is easily deduced from Lemma \ref{lem:p-adic}. Indeed we have
$$v_q\left(2^{|\alpha_*-\alpha_\dagger|}\pm 1\right)\leq u_q+\frac{\log |\alpha_*-\alpha_\dagger|}{\log q}<u_q+\frac{\log (52038 \log q)}{\log q}<u_q+1$$
since $q>52038 \log q$, which holds for $q>10^9$. The last statement of the lemma follows now from the observation that $v_q(2^x\pm 1)\geq u_q$ if $v_q(2^x\pm 1)\neq 0$.
\end{proof}

Now, we consider the four cases of Proposition \ref{prop:p=2} individually and show that the assumption that $q>10^9$ yields a contradiction in each case.
The problem is that an analogous statement of Lemma \ref{lem:alpha_beta_zero} does not hold for $p=2$
and we have to be careful when we repeat the proof of Proposition \ref{prop:Wief_padic} in the case that $p=2$.

\subsection{Case I} 

In this case we have 
\begin{align*}
 ab+1=&2q^{\beta}, &  bc+1=&2^{\alpha}q^{\beta_4}, \\
 ac+1=&2^{\alpha_2}q^{\beta}, &  bd+1=&2^{\alpha}q^{\beta_5}, \\
 ad+1=&2^{\alpha_3}q^{\beta}, &  cd+1=&2q^{\beta_6},
\end{align*}
with $1<\alpha<\alpha_2<\alpha_5$ and $\beta<\beta_4<\beta_5<\beta_6$. We consider the second equation of system \eqref{eq:SUnit} and obtain after dividing through the common denominator $2^\alpha q^\beta$ the
equation
\begin{equation}\label{eq:CaseI-p=2}
 q^{\beta_4-\beta}\left(q^{\beta_5-\beta}-1\right)=2^{\alpha_2}q^{\beta_5}-2^{\alpha_3}q^{\beta_4}+2^{\alpha_3-\alpha}-2^{\alpha_2-\alpha}.
\end{equation}
That is $2^{\alpha_2-\alpha}|q^{\beta_5-\beta}-1$ and Lemma \ref{lem:2-adic-bound} implies that $2^{\alpha_2-\alpha}<q^{1.51}$.%

Similarly as in the case that $p$ is odd, we observe that due to Lemma \ref{lem:s4_low_bound} we have 
$$2q^\beta \cdot 2^{\alpha}q^\beta=\gcd(s_1,s_4)\gcd(s_2,s_4)<s_4=2^{\alpha}q^{\beta_4},$$
i.e. $2\beta< \beta_4$. Moreover, $(ab+1)(ac+1)>bc+1$ implies that
$$2^{\alpha_2+1}q^{2\beta}>2^\alpha q^{\beta_4},$$
hence $2\beta+2>\beta_4$. Thus $\beta_4=2\beta+1$.

Next we compute similarly as in the case that $p$ is odd the quantity $a^2$ and we obtain:
\begin{align*}
 a^2 =& \frac{(s_1-1)(s_2-1)}{s_4-1}\\
 =& \frac{2^{\alpha_2+1}q^{2\beta}-2^{\alpha_2}q^{\beta}-2q^\beta+1}{2^\alpha q^{2\beta+1}-1}\\
 =& \frac{2^{\alpha_2-\alpha+1}}{q}+L\left(1.001\cdot \left(\frac{2^{\alpha_2}}{2^{2\alpha}q^{2\beta+2}}+\frac{2^{\alpha_2-\alpha}}{q^{\beta+1}}
+\frac{2}{2^\alpha q^{\beta+1}}+\frac{1}{2^\alpha q^{2\beta+1}}\right)\right)\\
 =& \frac{2^{\alpha_2-\alpha+1}}{q}+L\left(1.001\cdot \left(\frac{1}{2^{\alpha}q^{2\beta}}+\frac{2}{q^{\beta-1}}
 +\frac{1}{2^\alpha q^{\beta+1}}+\frac{1}{2^\alpha q^{2\beta+1}}\right)\right)
\end{align*}
If we assume that $\beta\geq 3$ we obtain
$$a^2=\frac{2^{\alpha_2-\alpha+1}}{q}+L\left(\frac{4.004}{q^2}\right)$$
and similarly as in Case I of the proof of Proposition \ref{prop:Wief_padic} we obtain a contradiction. Thus we may assume that $\beta\leq 2$.

The case that $\beta=0$ can be excluded since otherwise we would have that $s_1=ab+1=2$ and $a=b=1$ which is excluded. Assume that $\beta=1$, then a computation of $a^2$ as before shows that
$$a^2=\frac{2^{\alpha_2-\alpha+1}}{q}-\frac{2^{\alpha_2-\alpha}}{q^2}+L\left(1.001 \cdot \left(\frac{1}{2^{\alpha}q^{2}}+\frac{1}{2^{\alpha} q^{3}}
 +\frac{1}{2^{\alpha-1} q^{2}}+\frac{1}{2^\alpha q^{3}} \right)\right)
 $$
and since $2^\alpha\geq 4$ we conclude that 
$$a^2=\frac{2^{\alpha_2-\alpha+1}}{q}-\frac{2^{\alpha_2-\alpha}}{q^2}+L\left(\frac{0.8}{q^2}\right)$$
is an integer. This implies that
$$q^2\left| q2^{\alpha_2+1}-2^{\alpha_2}\right.$$
which is impossible.

By similar means we can show that in the case that $\beta=2$ no $\{2,q\}$-Diophantine quadruple exists.

\subsection{Case II}

In the second case we have 
\begin{align*}
 ab+1=&2q^{\beta'}, &  bc+1=&2^{\alpha_4}q^{\beta}, \\
 ac+1=&2^{\alpha}q^{\beta'}, &  bd+1=&2^{\alpha}q^{\beta_5}, \\
 ad+1=&2^{\alpha_3}q^{\beta}, &  cd+1=&2q^{\beta_6},
\end{align*}
with $0<\alpha<\alpha_3,\alpha_4$ and $\beta<\beta'<\beta_5<\beta_6$. We consider the second equation of system \eqref{eq:SUnit} and obtain 
\begin{equation}\label{eq:CaseII-p=2}
 q^{\beta'-\beta}\left(q^{\beta_5-\beta'}+1\right)=2^{\alpha}q^{\beta_5+\beta'-\beta}+2^{\alpha_3-\alpha}+2^{\alpha_3+\alpha_4-\alpha}q^\beta.
\end{equation}
Let $M=\min\{\alpha,\alpha_3-\alpha,\alpha_4-\alpha\}$, then $2^{M}|q^{\beta_5-\beta}+1$ and Lemma \ref{lem:2-adic-bound} implies that $M=1$.
In case that $\alpha=1$, we would obtain that $s_1=s_2$ a contradiction. If $\alpha_3=\alpha+1$ we obtain $2^\alpha q^{\beta'}=s_2<s_3=2^{\alpha+1}q^\beta$,
i.e. $q^{\beta'-\beta}<2$ an obvious contradiction. Finally in the case that $\alpha_4=\alpha+1$ we obtain a contradiction by considering the inequality
$s_2<s_4$.

\subsection{Case III}

In the third case we have:
\begin{align*}
 ab+1=&2^\alpha q^{\beta}, &  bc+1=&2^{\alpha_4}q^{\beta'}, \\
 ac+1=&2q^{\beta_2}, &  bd+1=&2q^{\beta_5}, \\
 ad+1=&2^{\alpha}q^{\beta'}, &  cd+1=&2^{\alpha_6}q^{\beta},
\end{align*}
with $1<\alpha\leq\alpha_4<\alpha_6$ and $\beta<\beta'<\beta_2<\beta_5$.
We consider the last equation of system \eqref{eq:SUnit} and we obtain that $2^{\alpha_4-\alpha}|q^{\beta'-\beta}-1$. Thus Lemma \ref{lem:2-adic-bound} implies that $2^{\alpha_4-\alpha}<2^{15}q$.
We apply Lemma \ref{lem:s4_low_bound} and obtain
$$q^{\beta}2^\alpha \cdot 2q^{\beta'}=\gcd(s_4,s_1)\gcd(s_4,s_2)<s_4=q^{\beta'}2^{\alpha_4}, $$
and therefore $q^\beta<2^{\alpha_4-\alpha-1}<2^{14}q$, which implies that $\beta=0,1$.

The case that $\beta=1$ is similar to the treatment of Case III, if $p$ is odd. However we may even assume that $2^{\alpha_4-\alpha}>2q$ and the inequality $s_2<s_3$ yields that $2^\alpha>2q$, that
is we have that $2^{\alpha_4}>4q^2$. Now, let us consider the second equation of system \eqref{eq:SUnit}. After rewriting the equation we obtain
$$2^{\alpha}(2^{\alpha_4-\alpha}+1)=q^{\beta'}2^{\alpha_4+\alpha}-4q^{\beta_2+\beta_4}+2q^{\beta_2}+2q^{\beta_5}.$$
Since $2^{\alpha_4-\alpha-1}+1<q^2$ we deduce that either $\beta'=1$ or $\beta_2=\beta'+1$. Note that $\beta=\beta'=1$ is excluded, hence we deduce that $\beta_2=\beta'+1$. Next, we compute  
\begin{align*}
 a^2=&\frac{(s_1-1)(s_2-1)}{s_4-1}\\
 =& \frac{2^{\alpha+1}q^{2\beta'+1}-2q^{\beta'+1}-2^{\alpha}q+1}{2^{\alpha_4}q^{\beta'}-1}\\
 =&  \frac{q^2}{2^{\alpha_4-\alpha-1}}+L\left(1.001 \cdot \left(\frac{q^2}{2^{2\alpha_4-\alpha-1}q^{\beta'}}+ \frac{q}{2^{\alpha_4-1}}+\frac{q}{2^{\alpha_4-\alpha}q^{\beta'}}
 +\frac{1}{2^{\alpha_4} q^{\beta'}}\right)\right)\\ 
 =&  \frac{q^2}{2^{\alpha_4-\alpha-1}}+L\left(\frac{0.51}{2^{\alpha_4-\alpha-1}}\right)
\end{align*}
which implies that $\frac{q^2}{2^{\alpha_4-\alpha-1}}$ is an integer, an obvious contradiction.

Therefore we have $\beta=0$ and by considering again the second equation of system \eqref{eq:SUnit} we obtain that $\beta'=1$. Note that the case that $\beta_2=\beta'+1$
has been excluded by the previous paragraph, i.e. we have that $\beta_2\geq 3$. Due to the inequality $2q^{\beta_2}=s_2<s_3=2^{\alpha}q$ we deduce that $2^{\alpha}>2q^{\beta_2-1}\geq 2q^2$. Next, we observe that 
$$\frac db=\frac{s_3-1}{s_1-1}=\frac{2^\alpha q-1}{2^\alpha-1}<1.001 q$$
and
$$\frac db=\frac{s_6-1}{s_4-1}=\frac{2^{\alpha_6}-1}{q2^{\alpha_4}-1}>\frac{2^{\alpha_6-\alpha_4}}q$$
and therefore $2^{\alpha_6-\alpha_4}<1.001q^2$. On the other hand we have that $c|\frac{s_6-s_4}{\gcd(s_6,s_4)}$, hence $c<2^{\alpha_6-\alpha_4}<1.001q^2$. This yields $2q^{\beta_2}=ac+1<c^2<2q^4$ and
we obtain that $\beta_2\leq 3$. Since we have excluded the case that $\beta_2< 3$ we are left with the possibility that $\beta_2=3$. Let us compute
$$ a^2=\frac{(s_1-1)(s_2-1)}{s_4-1}<\frac{s_1s_2}{s_4}=2^{\alpha+1-\alpha_4}q^2$$
and therefore we have $a<q\sqrt{2}$, since $\alpha_4\geq \alpha$. But $a<1.5q$ and $c<1.001q^2$ yield $2q^3=ac+1<2q^3$ a contradiction.

\subsection{Case IV}

In the last case we have:
\begin{align*}
 ab+1=&2^\alpha q^{\beta}, &  bc+1=&2q^{\beta_4}, \\
 ac+1=&2^{\alpha}q^{\beta'}, &  bd+1=&2^{\alpha_5}q^{\beta'}, \\
 ad+1=&2q^{\beta_3}, &  cd+1=&2^{\alpha_6}q^{\beta},
\end{align*} 
with $\alpha<\alpha_5<\alpha_6$ and $\beta<\beta'< \beta_3,\beta_4$.
In this case we consider the first equation of system \eqref{eq:SUnit} and obtain in combination with Lemma~\ref{lem:2-adic-bound} that $2^{\alpha_5-\alpha}<2^{15}q$. Applying Lemma \ref{lem:s4_low_bound}
we obtain
$$2^\alpha q^\beta \cdot 2 q^{\beta'}=\gcd(s_1,s_5)\gcd(s_3,s_5)<s_5=q^{\beta'}2^{\alpha_5},$$
hence $q^\beta<2^{\alpha_5-\alpha-1}<q^2$ and therefore $\beta=0,1$.

In the case that $\beta=1$ we proceed similarly as in the case that $p$ is odd. That is we consider the second equation of system \eqref{eq:SUnit}. 
After dividing through a common denominator and rearranging terms we obtain
\begin{equation}\label{eq:CaseIV-p=2}
2^\alpha(2^{\alpha_5-\alpha}+1)=q^{\beta'}p^{\alpha_5}-4q^{\beta_3+\beta_4-\beta'}+2q^{\beta_3-\beta'}+2q^{\beta_4-\beta'}.
\end{equation}
Since $2^{\alpha_5-\alpha}+1<q^2$ the $q$-adic valuation of the right hand side is at most $1$. Therefore we have either $\beta'=1$ or $\beta_4-\beta'=1$ or $\beta_3-\beta'=1$. Since $\beta=\beta'=1$
is a contradiction the first case cannot hold and the second case cannot hold due to an application of Lemma \ref{lem:s4_low_bound}. Indeed we obtain
$$2q^\beta \cdot 2q^{\beta'}=\gcd(s_1,s_4)\gcd(s_2,s_4)<s_4=2q^{\beta'+1}$$
which yields a contradiction since we assume that $\beta=1$.

Therefore we have $\beta_3=\beta'+1$ and $\beta_4\geq \beta'+2$. But $\beta_3=\beta'+1$ yields together with $s_2<s_3$ that $2^\alpha<2q$ and therefore $2^{\alpha_5}<2^{16}q^2<q^3$ and due to $s_4<s_5$ we
deduce that $\beta_4=\beta'+2$. Let us compute
\begin{align*}
 a^2=&\frac{(s_1-1)(s_2-1)}{s_4-1}\\
 =& \frac{2^{2\alpha}q^{1+\beta'}-2^{\alpha}q^{\beta'}-2^{\alpha}q+1}{2q^{\beta'+2}-1}\\
 =&  \left(2^{2\alpha}q^{1+\beta'}-2^{\alpha}q^{\beta'}\right)\left(\frac{1}{2q^{\beta'+2}}+\frac{1}{2q^{\beta'+2}(2q^{\beta'+2}-1)}\right)-
 (2^\alpha q-1)\frac{1}{2q^{\beta'+2}-1}\\
 =& \frac{2^{\alpha-1}(2^{\alpha}q-1)}{q^2}+\stackrel{X:=}{\overbrace{ \frac{2^\alpha(2^\alpha q-1)}{2q^2(2q^{\beta'+2}-1)}}}
 -\stackrel{Y:=}{\overbrace{ (2^\alpha q-1)\frac{1}{2q^{\beta'+2}-1}}}
\end{align*}
Since $2^\alpha<2q$ it is easy to see that $0<X,Y<\frac{1}{q^2}$ and therefore $|X-Y|<\frac{1}{q^2}$. Hence $a^2=\frac{2^{\alpha-1}(2^{\alpha}q-1)}{q^2}$ is an 
integer, which contradicts the fact that $q^2\nmid 2^{\alpha}(2^{\alpha}q-1)$.

Therefore we are left with the case that $\beta=0$. Again we consider \eqref{eq:CaseIV-p=2} and see that $q^M|2^{\alpha_5-\alpha}+1$, where $M=\min\{\beta',\beta_3-\beta',\beta_4-\beta'\}$.
Since $2^{\alpha_5-\alpha}+1<2^{15}q+1<q^2$ we obtain that $M\leq 1$.
Note that in case that $2^{\alpha_5-\alpha}<q-1$ we have $M=0$ which would yield an immediate contradiction. Therefore we have $M=1$ and $2^{\alpha_5-\alpha}\geq q-1$. However, $M=1$ implies
that either $\beta'=1$ or $\beta_4-\beta'=1$ or $\beta_3-\beta'=1$. The last two options have been excluded in our previous discussion, hence $\beta'=1$. Next we compute
$$ \frac da=\frac{s_5-1}{s_1-1}=\frac{2^{\alpha_5}q-1}{2^\alpha-1}<1.001 \cdot 2^{\alpha_5-\alpha}q<1.001 \cdot 2^{15} q^2$$
and
$$\frac da=\frac{s_6-1}{s_2-1}=\frac{2^{\alpha_6}-1}{2^\alpha q-1}>\frac{2^{\alpha_6-\alpha}}{q}.$$
Therefore we obtain that $2^{\alpha_6-\alpha}<1.001 \cdot 2^{15} q^3$. But since $2^{\alpha_5-\alpha}\geq q-1$ we obtain that $2^{\alpha_6-\alpha_5}<1.002 \cdot 2^{15} q^2$.
On the other hand we have $d|\frac{s_6-s_5}{\gcd(s_5,s_6)}$ and we obtain that $d<2^{\alpha_6-\alpha_5}<1.002 \cdot 2^{15} q^2$.

Next we estimate
\begin{equation}\label{eq:csquare}
c^2=\frac{(s_4-1)(s_2-1)}{s_1-1}>0.999 \frac{s_4s_2}{s_1}=1.998 q^{\beta_4+1}
\end{equation}
and we obtain 
$$1.01 \cdot 2^{30} q^4>d^2>c^2>1.998 q^{\beta_4+1}.$$
hence $1.02\cdot 2^{29} q^3>q^{\beta_4}$ and since we may assume that $q>10^9>1.02\cdot 2^{29}$ due to \cite[Theorem~1.3]{Szalay:2015} we deduce that $\beta_4\leq 3$. On the other
hand we have already excluded that $\beta_4<3$ and obtain that $\beta_4=3$. Moreover, \eqref{eq:csquare} yields $c>1.4 q^2$. On the other hand we have that
$c|\frac{s_4-s_2}{\gcd(s_4,s_2)}$, hence we get $c<q^2$ and we have a contradiction.

Since we excluded the existence of $\{2,q\}$-Diophantine quadruples for all four cases the proof that no $\{2,q\}$-Diophantine quadruple exists is complete.

\section{The $q$-adic Wieferich condition}\label{Sec:Wief_qadic}

This section is the $q$-adic analog of Section \ref{Sec:Wief_padic}. Thus we assume the divisibility condition $q^2 \nmid p^{q-1}-1$, i.e. $v_q(p^{q-1}-1)=1$. 
The main result of this section is the following proposition:

\begin{proposition}\label{prop:Wief_qadic}
Let $p<q$ be primes such that $p\equiv 3 \mod 4$ and $q^2 \nmid p^{q-1}-1$, i.e. $u_q=1$. If there exists a $\{p,q\}$-Diophantine quadruple, then $q\leq 700393$.
\end{proposition}

Similar as in the proof of Proposition \ref{prop:Wief_padic} we need a tool to estimate $q$-adic valuations. That is we prove the following $q$-adic variation of Lemma \ref{lem:p-adic-bound}:

\begin{lemma}\label{lem:q-adic-bound}
 Let $p<q$ be odd primes such that $p\equiv 3 \mod 4$ and let $*,\dagger \in \{0,1,2,3,4,5,6\}$. Moreover, set $\alpha_0=0$ and assume that $q>700393$. If $v_q(p^{q-1}-1)=u_q$, then we have that
 $v_q\left(p^{|\alpha_*-\alpha_\dagger|}\pm 1\right)=0,u_q$.  
\end{lemma}

\begin{proof}
 The lemma is a consequence of Lemma \ref{lem:p-adic} and Proposition \ref{prop:bound_AB}. Indeed we have 
 $$v_q\left(p^{|\alpha_*-\alpha_\dagger|}-1\right)<u_q+ v_q(|\alpha_*-\alpha_\dagger|)<u_q+\frac{\log (52038 \log q)}{\log q}<u_q+1,$$
 which implies the statement of the lemma provided that $\frac{\log (52038 \log q)}{\log q}<1$. Note that if $v_q(p^x-1)\neq 0$, then $v_q(p^x-1)\geq u_q$.
 
 To obtain the lemma we observe that the inequality $52038 \log q<q$ holds, if $q$ is a prime larger than $700393$. Indeed the inequality $52038 \log x<x$ holds if $x>700401$ and $700393$ is the largest prime
 less than $700401$.
\end{proof}

The proof of Proposition \ref{prop:Wief_qadic} is similar to the proof of Proposition \ref{prop:Wief_padic}. However, we have to discuss all four cases of Proposition \ref{prop:p=3mod4}.

\subsection{Case I}

In this case we have 
\begin{align*}
 ab+1=&q^{\beta}, &  bc+1=&p^{\alpha}q^{\beta_4}, \\
 ac+1=&p^{\alpha_2}q^{\beta}, &  bd+1=&p^{\alpha}q^{\beta_5}, \\
 ad+1=&p^{\alpha_3}q^{\beta}, &  cd+1=&q^{\beta_6},
\end{align*}
with $0<\alpha<\alpha_2<\alpha_5$ and $\beta<\beta_4<\beta_5<\beta_6$. We consider the second equation of system \eqref{eq:SUnit} and obtain after dividing through the common denominator $p^\alpha q^\beta$ the
equation
$$
 p^{\alpha_2-\alpha}(p^{\alpha_3-\alpha_2}-1)=p^{\alpha_3}q^{\beta_4}-p^{\alpha_2}q^{\beta_5}+q^{\beta_5-\beta}-q^{\beta_4-\beta}
$$
That is $q^{\beta_4-\beta}|p^{\alpha_3-\alpha_2}-1$ and Lemma \ref{lem:q-adic-bound} implies that $\beta_4-\beta=u_q=1$. We apply Lemma \ref{lem:s4_low_bound} and obtain
$$q^\beta \cdot p^\alpha q^\beta =\gcd(s_1,s_4)\gcd(s_2,s_4)<s_4=q^\beta \cdot p^\alpha q^\beta<p^\alpha q^{\beta+1}.$$
This implies that $q^\beta<q$, i.e. $\beta=0$. But, $\beta=0$ implies that $ab+1=1$, a contradiction.

\subsection{Case II}

In the second case we have 
\begin{align*}
 ab+1=&q^{\beta'}, &  bc+1=&p^{\alpha_4}q^{\beta}, \\
 ac+1=&p^{\alpha}q^{\beta'}, &  bd+1=&p^{\alpha}q^{\beta_5}, \\
 ad+1=&p^{\alpha_3}q^{\beta}, &  cd+1=&q^{\beta_6},
\end{align*}
with $0<\alpha<\alpha_3,\alpha_4$ and $\beta<\beta'<\beta_5<\beta_6$. We consider the first equation of system \eqref{eq:SUnit} and obtain after rearranging terms
$$
p^\alpha -1= p^{2\alpha}q^{\beta_5}-p^\alpha q^{\beta_5-\beta'}-q^{\beta_6}+q^{\beta_6-\beta'}.
$$
Lemma \ref{lem:q-adic-bound} yields $\beta_5-\beta'=1$. By an application of Lemma \ref{lem:s4_low_bound} we obtain
$$q^{\beta'} \cdot p^\alpha q^{\beta}=\gcd(s_1,s_5)\gcd(s_3,s_5)<s_5=q^{\beta'+1}p^\alpha, $$
which yields a contradiction unless $\beta=0$. But $\beta=0$ is also impossible due to Lemma \ref{lem:alpha_beta_zero}.

\subsection{The Case III}

In the third case we have:
\begin{align*}
 ab+1=&p^\alpha q^{\beta}, &  bc+1=&p^{\alpha_4}q^{\beta'}, \\
 ac+1=&q^{\beta_2}, &  bd+1=&q^{\beta_5}, \\
 ad+1=&p^{\alpha}q^{\beta'}, &  cd+1=&p^{\alpha_6}q^{\beta},
\end{align*}
with $0<\alpha\leq\alpha_4<\alpha_6$ and $\beta<\beta'<\beta_2<\beta_5$.
If we consider the first equation of system \eqref{eq:SUnit} and rearrange it in view of $q$-adic valuations we obtain
$$p^{\alpha_6-\alpha}+1=p^{\alpha_6}q^\beta-p^{\alpha_4}q^{2\beta'-\beta}+q^{\beta'-\beta}+p^{\alpha_4-\alpha}q^{\beta_4-\beta}.$$
Lemma \ref{lem:p-adic-bound} implies now that either $\beta=1$ or $\beta_2-\beta=1$. The second case
cannot hold since $\beta<\beta'<\beta_2$. However, if we consider the second equation of system \eqref{eq:SUnit} in view of $q$-adic valuations
we have
$$p^\alpha(p^{\alpha_4-\alpha}+1)=p^{\alpha+\alpha_4}q^{\beta'}-q^{\beta_2+\beta_5-\beta'}+q^{\beta_2-\beta'}+q^{\beta_5-\beta'}.$$
By Lemma \ref{lem:q-adic-bound} we obtain 
that either $\beta'=1$ or $\beta_2-\beta'=1$. Thus we conclude that $\beta=1$ and $\beta_2=\beta'+1$. Moreover the inequality $s_2<s_3$ yields that $p^\alpha>q$. Now by an almost identical computation
as in Case III in the proof of Proposition \ref{prop:Wief_padic} we obtain
\begin{align*}
 a^2 =&  \frac{q^2}{p^{\alpha_4-\alpha}}-\frac{q}{p^{\alpha_4}}+L\left(1.001 \cdot \left(\frac{q^2}{p^{2\alpha_4-\alpha}q^{\beta'}}+ \frac{q}{p^{2\alpha_4}q^{\beta'}}+\frac{q}{p^{\alpha_4-\alpha}q^{\beta'}}
 +\frac{1}{p^{\alpha_4} q^{\beta'}}\right)\right)\\ 
 =&  \frac{q^2}{p^{\alpha_4-\alpha}}-\frac{q}{p^{\alpha_4}}+L\left(\frac{1.001}{p^{\alpha_4-\alpha}} \cdot \left(\frac{1}{q}+ \frac{1}{q^3}+\frac{1}{q}+\frac{1}{q^3}\right)\right)\\
 =&  \frac{q^2}{p^{\alpha_4-\alpha}}-\frac{1}{p^{\alpha_4-\alpha}}\cdot \frac{q}{p^\alpha}+L\left(\frac{2.004}{p^{\alpha_4-\alpha}}\cdot \frac{1}{q}\right).
\end{align*}
and the same argument as in Case III in the proof of Proposition \ref{prop:Wief_padic} applies.

\subsection{The Case IV}

In the last case we have:
\begin{align*}
 ab+1=&p^\alpha q^{\beta}, &  bc+1=&q^{\beta_4}, \\
 ac+1=&p^{\alpha}q^{\beta'}, &  bd+1=&p^{\alpha_5}q^{\beta'}, \\
 ad+1=&q^{\beta_3}, &  cd+1=&p^{\alpha_6}q^{\beta},
\end{align*}
with $0<\alpha<\alpha_5<\alpha_6$ and $\beta<\beta'<\beta_3,\beta_4$.
We consider the third equation of system \eqref{eq:SUnit} in view of $q$-adic valuations and obtain
$$p^\alpha(p^{\alpha_6-\alpha}+1)=p^{\alpha+\alpha_6}q^{\beta}-q^{\beta_3+\beta_4-\beta}+q^{\beta_3-\beta}+q^{\beta_4-\beta} .$$
Lemma \ref{lem:q-adic-bound} yields that either $\beta=1$ or $\beta_3-\beta=1$ or $\beta_4-\beta=1$. 
The last two cases cannot hold, since $\beta<\beta'<\beta_3,\beta_4$. However, if we consider the second equation of system \eqref{eq:SUnit} we obtain
$$p^\alpha(p^{\alpha_5-\alpha}+1)=p^{\alpha+\alpha_5}q^{\beta'}-q^{\beta_3+\beta_4-\beta'}+q^{\beta_3-\beta'}+q^{\beta_4-\beta'}$$
and Lemma \ref{lem:q-adic-bound} implies that either $\beta'=1$ or $\beta_3-\beta'=1$ or $\beta_4-\beta'=1$. Obviously the first case cannot hold since $1=\beta<\beta'$. 

Let us assume for the moment that $\beta_4=\beta'+1$. By Lemma \ref{lem:s4_low_bound} we obtain
$$q \cdot q^{\beta'}=\gcd(s_1,s_4)\gcd(s_2,s_4) <s_4= q^{\beta'+1},$$
a contradiction. Therefore we may assume that $\beta_3=\beta'+1$ and $\beta_4=\beta'+2+\ell$ with some non-negative integer $\ell$.

Now, we are almost in the same situation as in Case IV of the proof of Proposition \ref{prop:Wief_padic}. Thus we compute the quantity $a^2$.
Before we do this let us note that $p^\alpha<q$ holds due to the inequality $s_2<s_3$. Now let us compute
\begin{align*}
 a^2=&\frac{(s_1-1)(s_2-1)}{s_4-1}\\
 =& \frac{p^{2\alpha}q^{1+\beta'}-p^{\alpha}q^{\beta'}-p^{\alpha}q+1}{q^{\beta'+2+\ell}-1}\\
 =&  \left(p^{2\alpha}q^{1+\beta'}-p^{\alpha}q^{\beta'}\right)\left(\frac{1}{q^{\beta'+2+\ell}}+\frac{1}{q^{\beta'+2+\ell}(q^{\beta'+2+\ell}-1)}\right)-
 \frac{p^\alpha q-1}{q^{\beta'+2+\ell}-1}\\
 =& \frac{p^{\alpha}(p^{\alpha}q-1)}{q^{2+\ell}}+\stackrel{X:=}{\overbrace{ \frac{p^\alpha(p^\alpha q-1)}{q^{2+\ell}(q^{\beta'+2+\ell}-1)}}}
 -\stackrel{Y:=}{\overbrace{\frac{p^\alpha q-1}{q^{\beta'+2+\ell}-1}}}
\end{align*}
Since $p^\alpha<q$ and $\beta'\geq 2$ it is easy to see that $0<X,Y<\frac{1}{q^{2+\ell}}$ and therefore $|X-Y|<\frac{1}{q^{2+\ell}}$. Hence $\frac{p^{\alpha}(p^{\alpha}q-1)}{q^{2+\ell}}$ is an 
integer, which contradicts the fact that $q^2\nmid p^{\alpha}(p^{\alpha}q-1)$.

Therefore the proof of Proposition \ref{prop:Wief_qadic} is complete. \hfill $\square$

Let us summarize what we have proved so far. Combining our results obtained in the $p$-adic case (Proposition \ref{prop:Wief_padic}) and our results found
in the $q$-adic case (Proposition \ref{prop:Wief_qadic}) we immediately obtain  

\begin{corollary}\label{cor:large_q}
  Let $p<q$ be primes and assume that $p\equiv 3 \mod 4$. Furthermore assume that $(p,q)$ is not an extreme Wieferich pair.
  Then a $S$-Diophantine quadruple exists only if $q\leq 700393$. Moreover if $p^2\nmid q^{p-1}-1$ we have 
  $$p<q<52038\log p.$$
\end{corollary}

\begin{proof}
 Note that if $p^2\nmid q^{p-1}-1$ we have $u_p=1<\max\left\{2,\frac{\log q}{\log p}\right\}$ and we can use the sharper bound provided by Proposition \ref{prop:Wief_padic}.
\end{proof}

\section{The remaining small cases}\label{Sec:Small}

In \cite{Szalay:2015} Szalay and the author found a method to reduce the huge bound for $\log d$ coming from the theory of linear forms in logarithms to
comparable small bounds by using continued fractions. In particular they proved the following lemma (see \cite[Lemma 3.1]{Szalay:2015}):

\begin{lemma}\label{lem:CFracMethod}
 Let $C\ge\log d$ and assume that for some real number $\delta>0$ we have
 \[\left|P\log p-Q\log q\right|>\delta\]
 for all convergents $P/Q$ to ${\log q}/{\log p}$ with $Q<{2C}/{\log q}$ and $P<{2C}/{\log p}$. 
 Then
 \[\log d<2C_1+u_q\log q+u_p\log p+\log\left(\frac{2C_1^2}{\log p\log q}\right),\]
 where
 \[C_1=\max\left\{\log\left(\frac{2}{\delta}\right),\log\left(\frac{8C}{\log p \log q}\right)\right\}.\]
\end{lemma}

Also the following lemma is useful (see \cite[Lemma 3.2]{Szalay:2015}):

\begin{lemma}\label{lem:alpha12_beta12_small}
Under the assumptions of Lemma~\ref{lem:CFracMethod}, $\alpha_1,\alpha_2<{C_1}/{\log p}$ and $\beta_1,\beta_2<{C_1}/{\log q}$ follows, where
\[C_1=\max\left\{\log\left(\frac{2}{\delta}\right),\log\left(\frac{8C}{\log p \log q}\right)\right\}.\]
Moreover, $\log(ab+1)<C_1$ and $\log(ac+1)<C_1$ also hold.
\end{lemma}

One can apply these two lemmas repeatedly and obtain small upper bounds for $\log d$ which makes a computer search feasible. Moreover the algorithm is very efficient
and it is possible to test for many pairs $(p,q)$ of primes whether a $\{p,q\}$-Diophantine quadruple exists. We use the following algorithm to find all
$S$-Diophantine quadruples for given $S=\{p,q\}$.

\begin{algo}\label{algo}
 Given two primes $p,q$ such that either $p\equiv 3\mod 4$ or $q\equiv 3 \mod 4$. Then the algorithm returns all possible $\{p,q\}$-Diophantine quadruples.
 \begin{enumerate}
  \item We compute the bound
  $$\log d<\log \left(p^Aq^B\right)<C_0:= 104076 \log p \log q.$$
  \item We use Lemma \ref{lem:CFracMethod} and Lemma \ref{lem:alpha12_beta12_small} to compute the new upper bounds $C$ and $C_1$ by using the upper bound $C_0\geq \log d$. 
  If $C<C_0 -0.1$ we put $C_0:=C$ and repeat this step.
  \item For all exponents $0\leq \alpha_1, \alpha_2\leq \frac{C_1}{\log p}$ and all exponents $0\leq\beta_1, \beta_2<\frac{C_1}{\log q}$
  we do the following:
  \begin{enumerate}
  \item Compute $g=\gcd(p^{\alpha_1}q^{\beta_1}-1,p^{\alpha_2}q^{\beta_2}-1)$. Note that $\gcd(s_1-1,s_2-1)=\gcd(ab,ac)=a\gcd(b,c)$.
  \item If $g>0$ we compute for all divisors $a$ of $g$ the quantity $b=\frac{p^{\alpha_1}q^{\beta_1}-1}a$ and if $a<b$ we compute $c=\frac{p^{\alpha_2}q^{\beta_2}-1}a$
  and check whether $c$ is an integer such that $a<b<c$. If $(a,b,c)$ is a $\{p,q\}$-Diophantine triple, i.e. the only prime divisors of $bc+1$ are $p$ and $q$, then we store $(a,b,c)$ in a list $\mathrm{L}$.
 \end{enumerate}
  \item For all exponents $0\leq \alpha_6 \leq \frac{C}{\log p}$, all exponents $0\leq\beta_6<\frac{C}{\log q}$ and all triples $(a,b,c)\in\mathrm{L}$
  we do the following: 
  \begin{enumerate}
  \item We compute $d=\frac{p^{\alpha_6}q^{\beta_6}-1}{c}$.
  \item We check whether $d$ is an integer such that $d>c$.
  \item We check whether $(a,b,c,d)$ is a $\{p,q\}$-Diophantine quadruple, that is we check whether the only prime divisors of $ad+1$ and $bd+1$ are $p$ and $q$.
  \item If $(a,b,c,d)$ is a $\{p,q\}$-Diophantine quadruple, we store $(a,b,c,d)$ in a list $\mathrm{Quad}$.
  \end{enumerate}
  \item We return the list $\mathrm{Quad}$.
 \end{enumerate}
\end{algo}

We implemented this algorithm in PARI/GP \cite{pari} and checked all pairs of primes such that $p<q$, $p\equiv 3 \mod 4$ and $q<52038 \log p$. This are $340306885$ pairs and it
took about $21$ on a usual PC (Intel i7-7500U -- 2.70 GHz). However we found no $S$-Diophantine quadruple. In view of Corollary \ref{cor:large_q} this proves Theorem \ref{th:gen} for all
pairs $(p,q)$ of primes such that $p<q$, $p\equiv 3\mod 4$ and $p^2\nmid q^{p-1}-1$.

For the remaining cases we do the following. For all pairs of primes $(p,q)$ such that $p<q$, $p\equiv 3 \mod 4$ and $52038 \log p\leq q \leq 700393$, we check whether $p^2| q^{p-1}-1$.
In the cases for which $p^2| q^{p-1}-1$ we apply Algorithm \ref{algo}. Let us note that only $24297$ pairs $(p,q)$ of the $60321782$ remaining pairs of primes satisfy $p^2| q^{p-1}-1$.
Therefore the running time of $87$ seconds was rather short. Since we found in this second round no $S$-Diophantine quadruple the proof of Theorem \ref{th:gen} is now complete.

\section{Further Remarks and open problems}\label{Sec:Remarks}

In this final section we want to discuss several open problems concerning this topic. First, we want to mention that with some effort it seems
to be possible to resolve the case of primes $p<q$ such that $q\equiv 3 \mod 4$. We hope to be able to prove in a forthcoming paper the following conjecture:

\begin{conjecture}
 Let $p<q$ be primes and assume that $(p,q)$ is not a Wieferich pair nor satisfies $p\equiv q\equiv 1 \mod 4$. Then there is no $\{p,q\}$-Diophantine quadruple.
\end{conjecture}

It would be very interesting to get rid of the Wieferich condition as we were able do to in the case that $S=\{2,q\}$ or $S=\{3,q\}$. With a little more effort will also prove
in a forthcoming paper the following conjecture:

\begin{conjecture}\label{th:p-5}
 Let $q\not\equiv 1 \mod 4$ be a prime. Then there is no $\{5,q\}$-Diophantine quadruple.
\end{conjecture}

However, with much effort such results as Theorem \ref{th:p23} or Conjecture \ref{th:p-5} could also be established with $p=7$ or even $p=11$. But the problem is that
to the authors knowledge for fixed $p$ we do not know how large $u_p=v_p(q^{p-1}-1)$ can get. To the authors knowledge the best known upper bound for
$u_p$ is due to Yamada \cite{Yamada:2010} who used the very sharp results due to Bugeaud and Laurent~\cite{Bugeaud:1996} for linear forms in two $p$-adic logarithms.
In particular, Yamada obtained that
$$u_p\leq \left\lfloor 283(p-1)\frac{\log 2}{\log p}\cdot\frac{\log 2q}{\log p}\right\rfloor+4.$$
However this bound seems to be far from optimal. In particular, Yamada \cite[Conjecture 1.3]{Yamada:2010} conjectures that 
$$u_p\leq 2 +\frac{\log q+\log\log q+\log\log p}{\log p}.$$
In view of this conjecture it seems very unlikely that our $p$-adic Wieferich condition that $u_p<\max\left\{2,\frac{\log q}{\log p}\right\}$ is not fulfilled
if $q$ is large compared to $p$ and in view of our definition of an extreme Wieferich pair we are interested in the following problem:

\begin{problem}
Does there exist an extreme Wieferich pair $(p,q)$, with $q>p^2$? 
\end{problem}

The author's guess is that such an extreme Wieferich pair does not exist. Nevertheless to prove a theorem without a Wieferich type
criterion using the methods presented in this paper we would need to show that $v_p(q^{p-1}-1)<\frac{c\log q}{\log p}$, where
$c$ is a small (e.g. $c<2$) absolute constant, a result that seems to be far out of reach. 

It would be also interesting to get rid of the congruence condition that either $p$ or $q$ is $\equiv 3 \mod 4$. In particular, it would be interesting to prove the following
weaker form of Conjecture \ref{con:S-quadruple}:

\begin{conjecture}
 Let $p<q$ be primes and assume that $(p,q)$ is not a Wieferich pair. Then there exists no $\{p,q\}$-Diophantine quadruple.
\end{conjecture}


\def\cprime{$'$}

\end{document}